\documentclass[reqno]{amsart}

\usepackage[english]{babel}

\usepackage[a4paper,top=2.9cm,bottom=2.9cm,left=3cm,right=3cm,marginparwidth=1.75cm]{geometry}

\usepackage{amssymb}
\usepackage{amsmath}
\usepackage{graphicx}
\usepackage{tikz-cd}
\usepackage{amsthm}
\usepackage{aliascnt}
\usepackage{enumitem}
\usepackage{hyperref}
\usepackage[capitalise]{cleveref}

\newtheorem{theorem}{Theorem}[section]
\newtheorem*{Thm*}{Theorem}

\newaliascnt{corollary}{theorem}
\newtheorem{corollary}[corollary]{Corollary}
\aliascntresetthe{corollary}

\newaliascnt{proposition}{theorem}
\newtheorem{proposition}[proposition]{Proposition}
\aliascntresetthe{proposition}

\newaliascnt{lemma}{theorem}
\newtheorem{lemma}[lemma]{Lemma}
\aliascntresetthe{lemma}

\theoremstyle{definition}

\newaliascnt{observation}{theorem}

\aliascntresetthe{observation}

\newaliascnt{claim}{theorem}

\aliascntresetthe{claim}

\newaliascnt{example}{theorem}
\newtheorem{example}[example]{Example}
\aliascntresetthe{example}

\newaliascnt{definition}{theorem}
\newtheorem{definition}[definition]{Definition}
\aliascntresetthe{definition}

\newaliascnt{remark}{theorem}
\newtheorem{remark}[remark]{Remark}
\aliascntresetthe{remark}

\crefname{corollary}{Corollary}{Corollaries}
\crefname{proposition}{Proposition}{Propositions}
\crefname{lemma}{Lemma}{Lemmas}
\crefname{observation}{Observation}{Observations}
\crefname{claim}{Claim}{Claims}
\crefname{example}{Example}{Examples}
\crefname{definition}{Definition}{Definitions}
\crefname{remark}{Remark}{Remarks}

\makeatletter
\newcount\pb@x \newcount\pb@y
\def\pb@scan#1{%
  \ifx#1\relax
  \else
    \if#1r\advance\pb@x\@ne\fi
    \if#1l\advance\pb@x\m@ne\fi
    \if#1d\advance\pb@y\@ne\fi
    \if#1u\advance\pb@y\m@ne\fi
    \expandafter\pb@scan
  \fi}
\def\pb@sign#1{\ifnum#1>\z@ 1\else\ifnum#1<\z@ -1\else 0\fi\fi}
\newcommand{\pb}{\@ifnextchar[{\pb@atpos}{\pb@atpos[0.15]}}
\def\pb@atpos[#1]#2{%
  \pb@x\z@ \pb@y\z@ \pb@scan#2\relax
  \pgfmathparse{atan2(-(\pb@sign{\pb@y}),\pb@sign{\pb@x})+45}%
  \edef\pb@corner{\noexpand\arrow[phantom, pos=#1]{#2}%
    {\noexpand\rotatebox[origin=c]{\pgfmathresult}{$\noexpand\lrcorner$}}}%
  \pb@corner}
\makeatother

\makeatletter
\newcommand{\switchingdpb}[1]{%
  \in@{l}{#1}%
  \ifin@
    \triangleright
  \else
    \triangleleft
  \fi
}
\makeatother

\newcommand{\cart}[1]{\arrow[phantom]{#1}{\dagger}}

\def\Set{\mathsf{Set}}

\def\id{\mathsf{id}}
\def\el{\mathsf{el}}
\def\cod{\mathsf{cod}}
\def\dom{\mathsf{dom}}

\def\DFib{\mathsf{DFib}}

\def\Sq{\mathsf{Sq}}
\def\HKl{\mathsf{HKl}}
\def\Span{\mathsf{Span}}
\def\Rel{\mathsf{Rel}}
\def\Mod{\mathsf{Mod}}
\def\Mat{\mathsf{Mat}}
\def\Cat{\mathsf{Cat}}
\def\thg{\mathord{\text{--}}}
\def\Top{\mathsf{Top}}
\def\Poly{\mathsf{Poly}}

\def\KK{\HKl_{\mathsf{T}}(\E)}

\def\Cnr{\mathsf{Cnr}}

\newcommand{\tor}{\ensuremath{\relbar\joinrel\mapstochar\joinrel\rightarrow}}
\def\A{\mathcal{A}}
\def\B{\mathcal{B}}
\def\C{\mathcal{C}}
\def\D{\mathcal{D}}
\def\K{\mathcal{K}}
\def\L{\mathcal{L}}
\def\M{\mathcal{M}}

\def\P{\mathcal{P}}

\def\sec\S
\def\S{\mathcal{S}}
\def\E{\mathcal{E}}
\def\X{\mathcal{X}}
\def\Y{\mathcal{Y}}

\def\op{{\mathsf{op}}}

\def\To{\Rightarrow}
\def\bt{\mathbin{\bowtie}}
\def\defeq{\mathbin{:=}}

\newcommand{\hcolor}{orange}
\newcommand{\kcolor}{blue}
\newcommand{\idarr}[1]{\arrow[equal]{#1}}

\newcommand{\arrowt}[2]{\arrow[#1, "|"{marking, scale=0.7}, "#2"{outer sep=3pt}]}
\newcommand{\arrowts}[2]{\arrow[#1, swap, "|"{marking, scale=0.7}, "#2"{outer sep=3pt}]}

\newcommand{\harrow}[2]{\arrow[\hcolor]{#1}{#2}}
\newcommand{\karrow}[2]{\arrow[\kcolor]{#1}{#2}}

\newcommand{\karrowl}[2]{\arrow[#1, "|"{marking, scale=0.7}, leftarrow, \kcolor, "#2"{outer sep=3pt}]}
\newcommand{\karrowsl}[2]{\arrow[#1, "|"{marking, scale=0.7}, leftarrow, swap, \kcolor, "#2"{outer sep=3pt}]}
\newcommand{\harrows}[2]{\arrow[\hcolor, swap]{#1}{#2}}
\newcommand{\karrows}[2]{\arrow[\kcolor, swap]{#1}{#2}}
\newcommand{\harrowm}[2]{\arrow[#1, \hcolor, "#2" description]}
\newcommand{\karrowm}[2]{\arrow[#1, \kcolor, "#2" description]}

\tikzcdset{cell label/.style={phantom, start anchor=center, end anchor=center}}

\newcommand{\hcell}[2]{\arrow[\hcolor, cell label]{#1}{#2}}
\newcommand{\kcell}[2]{\arrow[\kcolor, cell label]{#1}{#2}}

\newcommand{\cell}[2]{\arrow[cell label]{#1}{#2}}

\newcommand{\hcorner}[4]{#1 \harrow{r}{#2} \arrow[swap]{d}{#3} \hcell{dr}{#4}}
\newcommand{\kcorner}[4]{#1 \karrow{r}{#2} \arrow[swap]{d}{#3} \kcell{dr}{#4}}
\newcommand{\corner}[4]{#1 \arrow{r}{#2} \arrow[swap]{d}{#3} \cell{dr}{#4}}

\newcommand{\cone}{\omega_1}
\newcommand{\ctwo}{\omega_2}
\newcommand{\cthree}{\omega_3}
\newcommand{\cfour}{\omega_4}

\newcommand{\wone}{W1}
\newcommand{\wtwo}{W2}
\newcommand{\wthree}{W3}
\newcommand{\wfour}{W4}
\newcommand{\wfive}{W5}

\title{Composing Double Categories}
\author{Mitchell Buckley}
\email{mitch.buckley@gmail.com}
\author{Richard Garner}
\email{richard.garner@mq.edu.au}
\address{Centre of Australian Category Theory, Macquarie University,
  NSW 2109, Australia}

\begin{document}

\begin{abstract}
  Many kinds of generalised syntax, including internal categories,
  $\mathsf{T}$-categories or polynomial monads, arise as monads in
  bicategories of a certain flavour: maps are some kind of span, and
  composition is some kind of pullback. We explain the pattern behind
  these similarities: each bicategory is stitched together from two
  simpler pieces by a distributive law, where the distributivity is
  governed by pullback.

  A key novelty is in the kind of distributive law we consider; where
  Rosebrugh and Wood considered distributive laws between categories,
  we instead consider pseudo-distributive laws between double
  categories. A key construction, drawing on work of Weber, obtains
  such a distributive law from a so-called \emph{co-crossed} double
  category: one whose codomain functor is a fibration, and whose
  cartesian cells are stable under horizontal composition.

  \looseness=-1 We show the horizontal Kleisli double category of a
  monad $\mathsf{T}$ is co-crossed just when $\mathsf{T}$ is
  cartesian, and that the composite double category recovers
  $\mathsf{T}$-spans; that the double category associated to an
  identity-on-objects functor is co-crossed just when that functor is
  a calibration in the sense of B\'enabou; and that Street's
  calibrated bicategories yield co-crossed double categories in a
  similar manner, recovering the double category of polynomials of
  Gambino--Kock.
\end{abstract}

\dedicatory{Dedicated to Ross Street on the occasion of his 80th birthday}

\maketitle

\section{Introduction}

This paper was born out of questions relating to syntax in computer
science: an apparently unpropitious start for a paper dedicated to
Ross, in whose work ``syntax'' and ``computer science'' can hardly be
said to be prominent themes. However, if we ask not about syntax
itself, but about the different \emph{kinds of syntax} we may work
with, then our dedication becomes much more plausible. While
common-or-garden syntax is encoded by free monads on sets, more
nuanced kinds use notions such as operads, polynomial monads, or the
$\mathsf{T}$-categories of Burroni~\cite{Burroni1971T-categories}---structures whose common feature
is that they are expressible as monads in a suitable
\emph{bicategory}.

Thus, we have the bicategory $\Span(\E)$ of spans in a
category with pullbacks, wherein monads are internal categories; the
bicategory of \emph{$\mathsf{T}$-spans}~\cite{Burroni1971T-categories,
  Leinster2004Higher}, wherein monads are $\mathsf{T}$-categories; and
the bicategory of \emph{polynomials} in a locally cartesian closed
category~\cite{GK2013}, whose monads are polynomial monads. In each
case, a $1$-cell is a kind of span, and composition is a kind of
pullback; yet the details seem to pull against this high-level
similarity. For example, composing $\mathsf{T}$-spans involves a
pullback followed by the multiplication of $\mathsf{T}$, while
composing polynomials involves a \emph{distributivity pullback}, built
using dependent products.

Our goal is to explain how all these apparent differences dissolve
with the right viewpoint: we will see that each of our candidate
bicategories can be stitched together from simpler pieces using the
\emph{distributive laws} of Beck~\cite{Beck1969Distributive}, where in
every case, the distributions in question are the pullbacks noted
above.

For monads $\mathsf{S}$ and $\mathsf{T}$ on a category $\C$, a
distributive law comprises a natural transformation
$\delta \colon TS \Rightarrow ST$, compatible with the units and
multiplications of $\mathsf{S}$ and $\mathsf{T}$; Beck's theory then
asserts, among other things, that the composite functor $ST$
underlies a \emph{composite monad} $\mathsf{ST}$. The motivating
example takes $\mathsf{S}$ to be the monad for abelian groups (written
additively), $\mathsf{T}$ the monad for monoids (written
multiplicatively), and for $\delta \colon TS \Rightarrow ST$, the
operation which rewrites a product of sums as a sum of products. The
composite monad $\mathsf{S}\mathsf{T}$ is then the monad for rings.

As Ross noted in~\cite{Street1972The-formal}, the theory of
distributive laws can be worked out for monads in any bicategory; and
in the very pretty paper~\cite{rosebrugh_distributive_2002}, Rosebrugh
and Wood do this for monads in the bicategory $\Span(\Set)$. In this
bicategory, monads on $O$ are categories with set of objects $O$, and
a distributive law between two such monads $\E$ and $\M$ is an
operation which \emph{commutes $\M$-maps past $\E$-maps}:
\begin{equation}
  \label{eq:swap}
  \begin{tikzcd}
    A  \harrow{r}{m \in \M} & B \karrow{d}{e \in \E}\\
    & C
  \end{tikzcd} \quad \qquad \mapsto \qquad \quad \begin{tikzcd}
    A \karrows{d}{{e_m \in \E}} \\
    B' \harrows{r}{{m_e \in \M}} & C\rlap{ .}
  \end{tikzcd}
\end{equation}
In this case, the composite monad $\M\E$ is a category $\C$ with the
same object-set, in which the morphisms are composable pairs
\begin{equation*}
  \begin{tikzcd}
    A \karrow{r}{e \in \E} & B \harrow{r}{m \in \M} & C\rlap{ ,}
  \end{tikzcd}
\end{equation*}
and composition places the pairs end-to-end, applies~\eqref{eq:swap},
and then composes in $\E$ and $\M$ separately. In this way,
distributive laws in $\Span(\Set)$ allow us to build categories by
stitching together simpler ones; for example, the
category $\Delta_+$ of order-preserving maps between finite ordinals
arises from a distributive law between the categories
$(\Delta_+)_{\mathrm{inj}}$ and $(\Delta_+)_{\mathrm{surj}}$ of
order-preserving \emph{injections}, resp., \emph{surjections}, between
finite ordinals.

The main conceptual contribution of this paper is to show that the
same kind of stitching-together allows us to build the bicategories of
$\mathsf{T}$-spans and of polynomials. For $\mathsf{T}$ a cartesian
monad on a category $\E$ with pullbacks, a $\mathsf{T}$-span from $a$
to $b$ is a span $f \colon Ta \leftarrow x \rightarrow b \colon g$,
which we can decompose as a composite
\begin{equation*}
  \begin{tikzcd}
    & x \arrow[swap]{dl}{f} \arrow{dr}{g} & \\
    Ta & & b
  \end{tikzcd} \qquad \cong \qquad
  \underbrace{\begin{tikzcd}[ampersand replacement=\&]
    \& x \arrow[swap]{dl}{\eta_x} \arrow{dr}{g} \& \\
    Tx \& \& b
  \end{tikzcd}}_{\in \E} \quad \circ \quad 
  \underbrace{\begin{tikzcd}[ampersand replacement=\&]
    \& x \arrow[swap]{dl}{f} \arrow{dr}{1_x} \& \\
    Ta \& \& x
  \end{tikzcd}}_{\in \E_\mathsf{T}}
\end{equation*}
whose two components are specified by a map $g$ in $\E$ and a map $f$
in the Kleisli category $\E_\mathsf{T}$. Thus, composing two
$\mathsf{T}$-spans requires us to commute a map of $\E$ past a Kleisli
map, which we may do by pulling back. In a similar way, a polynomial
$I \leftarrow A \rightarrow B \rightarrow J$ can be seen as a span
$I \leftarrow A \rightarrow B$ of $\E$ followed by a map $B \rightarrow J$ in
$\E$; so that to compose polynomials, we must commute a map past a
span. This, too, can be done by pullback; more precisely, as Ross
showed in~\cite{Street2020}, by a bipullback in the bicategory
$\Span(\E)$.

Now, the distributive laws of~\cite{rosebrugh_distributive_2002} can
only stitch together categories to yield other categories, so to
obtain bicategories, we must generalise. The obvious move is to
observe that bicategories are \emph{pseudomonads} in the weak
$3$-category $\Mat(\Cat)$, whose objects are sets and whose homs are
the $2$-categories $\Cat^{A \times B}$ of $\Cat$-valued matrices, and
to use \emph{pseudo-distributive
  laws}~\cite{cheng_pseudo-distributive_2003,
  Marmolejo1999Distributive} between such pseudomonads. Unfortunately,
this does not help. The pieces we wish to stitch together, such as
$\E$ and $\E_\mathsf{T}$, are mere $1$-categories; and a
pseudo-distributive law between $1$-categories \emph{qua} pseudomonads
in $\Mat(\Cat)$ is necessarily an honest distributive law, as there
are no non-trivial $3$-cells around to make it otherwise.

It turns out that we can get past this obstacle by moving from
bicategories to the (weak) \emph{double categories}
of~\cite{Grandis1999Limits}. These are pseudomonads in the weak
$3$-category $\Span(\Cat)$; they enhance bicategories with a second
direction of $1$-cell, which composes strictly, and a more general
kind of $2$-cell bounded by a square of $1$-cells of both kinds. The
reason this helps is that a $1$-category $\C$ can become a double
category in a non-trivial way as its \emph{double category of squares}
$\Sq(\C)$, whose horizontal and vertical maps are those of $\C$, and
whose cells are commuting squares; and by stitching such double
categories together, we get double categories which are again
non-trivial and in particular, have non-trivial underlying
bicategories.

This ties nicely back to~\cite{rosebrugh_distributive_2002}: therein,
Rosebrugh and Wood show that distributive laws between categories
correspond to \emph{strict} factorisation systems on the composite
category. In asking whether non-strict factorisation systems can be
captured by some weakening of the equations of a distributive law,
they note that they are ``\emph{lacking a higher categorical dimension
  in which to replace equations between transformations by
  isomorphisms}''. Our setting supplies the missing piece, which turns
out to be not a higher dimension, but an \emph{orthogonal} one---
indeed, as we will see below, this extra ingredient is sufficient for
us to capture arbitrary orthogonal factorisation systems in our
setting, to within equivalence.

It is a non-trivial task to describe a double-categorical distributive
law; thus, in order to do so for the bicategories of interest, we
eschew an ad hoc approach, and instead describe a single mechanism
that applies to them all. The input to this mechanism is a
\emph{co-crossed} double category $\C$: one whose codomain functor
$\cod \colon \C_1 \to \C_0$ is a fibration, and whose cartesian cells
are closed under nullary and binary horizontal composition. Given a
horizontal arrow $u \colon c \to b$ and a vertical arrow
$f \colon a \to b$ of $\C$, the cartesian lift of $u$ at $f$ is a cell
\begin{equation*} \begin{tikzcd} d
\arrow{r}{u_f} \arrow[swap]{d}{f_u} \cart{dr} & a \arrow{d}{f} \\ c
\arrow[swap]{r}{u} & b\rlap{ ,} \end{tikzcd} \end{equation*}
whose other edges $u_f$ and $f_u$ are precisely the data needed to commute $f$ past
$u$. In this way, a co-crossed $\C$ yields a distributive law between
(the horizontal opposite of) $\C$ and $\Sq(\C_0)$ (\cref{prop:fibred}); we call
the composite double category the \emph{double category of corners} $\Cnr \C$. For
example, the cartesian cells of $\Sq(\E)$ are the pullback squares, so
that $\Sq(\E)$ is co-crossed just when $\E$ has pullbacks, and its
double category of corners is then the double category 
$\Span(\E)$ (\cref{cor:spans}).

The above construction is a weakening of one due to
Weber~\cite[Construction~5.3.3]{weber2015}, who obtains a distributive
law from a \emph{crossed internal category} in a $2$-category, and
terms its composite the \emph{$2$-category of corners}. When
specialised to the $2$-category $\Cat^{\mathrm{co}}$, Weber's crossed internal
categories are strict double categories whose codomain functor is a
split fibration, with chosen cartesian lifts strictly preserved by
horizontal composition; continuing along the path
of~\cite{stepan2023}, we relax all three conditions. The weakening is
essential; the coherence cells in the composite double categories
involve the comparison isomorphisms between different cartesian lifts,
and when these are non-trivial, a genuine bicategory emerges.

Our first main application is to $\mathsf{T}$-spans. For a monad
$\mathsf{T}$ on $\E$, the \emph{horizontal Kleisli double category}
$\KK$ has the same objects and vertical arrows as $\E$, horizontal
arrows given by Kleisli maps $a \to Tb$, and commuting squares in $\E$
as cells. We show (Theorems~\ref{thm:tspans} and~\ref{thm:T-cart}):
\begin{Thm*}
  Let $\mathsf{T}$ be a monad on a category $\E$. The double category
  $\KK$ is co-crossed just when $\E$ has pullbacks and
  $\mathsf{T}$ is cartesian; the double category of
  corners is then that of $\mathsf{T}$-spans.
\end{Thm*}
\noindent The
``only if'' is notable: it tells us that the hypotheses usually imposed so
that $\mathsf{T}$-spans compose associatively are \emph{exactly} those which
make our mechanism apply.

Our second main application is to polynomials in a locally cartesian
closed category $\E$, which will be obtained by applying our mechanism
to a suitable double category of spans in $\E$. In the early stages of
this work, we defined this double category and verified its
co-crossedness by hand; however, the second author then recalled some
recent work of Ross~\cite{Street2020}, which constructed the
bicategory of polynomials as a bicategory of ``spans in spans''. A
central organising principle in this work is the fact that $\Span(\E)$
is an example of what Ross terms a \emph{calibrated bicategory}, in a
nod to B\'enabou's notion of a \emph{calibrated
  category}~\cite{benabou1975theories}. We were surprised to find that
our theory and the theory of calibrations fit together like a glove.

To explain how, we start one dimension down. A
calibration of a category $\M$ is an identity-on-objects functor
$J \colon \A \to \M$, the maps in whose image we term \emph{neat},
such that:
\begin{enumerate}[label=(R\arabic*),itemsep=0em]
\item neat maps are left-cancellable: if $q$ and $qp$ are neat, then
  so is $p$;
\item neat maps are stable under pullback along arbitrary maps of $\M$.
\end{enumerate}
(We are being slightly loose here for the sake of a clean exposition;
see \cref{def:calibration} for the details).
Associated to $J$ is a double category $\mathsf{H}J$
whose vertical arrows are the neat maps, whose horizontal arrows are
all maps of $\M$, and whose cells are commuting squares; and we have:
\begin{Thm*}
  An identity-on-objects $J \colon \A \to \M$ is a calibration
  if and only if $\mathsf{H}J$ is co-crossed.
\end{Thm*}
\noindent (See \cref{lemma:cocrossd-calib1}.) In particular,
$\mathsf{T}$ is cartesian and $\E$ has pullbacks just when the free
functor $\E \to \E_\mathsf{T}$ is a calibration
(\cref{thm:cart-equiv}).

When $J$ is a calibration, we can construct the double category of
corners of the associated co-crossed $\mathsf{H}J$; this will be a
bicategory of spans with one neat leg, where (R2) allows horizontal
composition to proceed. For example, the bicategories of partial maps,
of topological correspondences~\cite{Miller2024}, and of spans of
categories with left leg a discrete fibration, all arise from
calibrations.

There is something quite striking about the correspondence between
calibrations and co-crossed double categories. A priori, only
condition (R2) is needed to form the bicategory of spans with one neat
leg; (R1) seems incidental, and rather strong. However, when
transported to the double category $\mathsf{H}J$, (R1) becomes
precisely the condition that its horizontal identity maps are
cartesian. Thus, the extra condition on a calibration identifies
precisely when the composition of spans with one neat leg arises as
the composite of a distributive law.

Ross's calibrated bicategories are an analogue of B\'enabou's
calibrations, one dimension up. $\M$ becomes a bicategory and $J$ an
identity-on-objects pseudofunctor; pullbacks become bipullbacks; and
there is (loosely speaking) only one extra condition, namely, that
neat maps be \emph{groupoid fibrations} in $\M$. As before, we are
then able to show that for such a $J$, the double category
$\mathsf{H}J$, suitably generalised to the new setting, is
co-crossed (\cref{prop:HHcocrossed}). The motivating example is the
calibration of $\Span(\E)$ by $\E$, for $\E$ locally cartesian closed,
whose bipullbacks are constructed using dependent products, just as in
the composition of polynomials; and using this, we obtain:

\begin{Thm*} For $\E$ locally cartesian closed, the
identity-on-objects pseudofunctor $(\thg)_* \colon \E \to \Span(\E)$
sending $f$ to $(\id, f)$ is a calibration, and the double category of
corners of $\mathsf{H}((\thg)_*)$ is equivalent to the double category
$\Poly(\E)$ of polynomials in $\E$. \end{Thm*}
\noindent (See \cref{prop:spans-calibration,prop:poly}.) The same
mechanism also reconstructs the other main examples
from~\cite{Street2020} (\cref{ex:rel,ex:mod}). A closely related
approach to $\Poly(\E)$ was described by Arkor and
Clarke~\cite{narkor2025}, who call a double category \emph{bridgeable}
when its domain functor is a fibration and its horizontal identities
are cartesian, and state that the double category
$\Span(\E)^{\mathrm{retro}}$ of spans and retrocells~\cite{retrocells}
is bridgeable, with associated double category of corners---or, in
their nomenclature, \emph{double category of bridges}---given by
$\Poly(\E)$.

We note two limits to our results. First, unlike in the categorical
case, the converse of \cref{prop:HHcocrossed} fails: $\mathsf{H}J$
may be co-crossed without $J$ being a bicategorical calibration
(\cref{rem:HHconverse}). Second, Ross constructs polynomials not only
in locally cartesian closed categories, but in any category with
pullbacks, at the cost of restricting the permitted spans to his
\emph{right lifters}; our framework could accommodate this, as we
explain in \cref{rem:Jaxioms}, but we do not pursue this here.

\textbf{Overview}. In \cref{sec:double-categories}, we recall double
categories and their presentation as pseudomonads in $\Span(\Cat)$. In
\cref{sec:distributive_laws} we define distributive laws between
double categories as pseudo-distributive laws between the
corresponding pseudomonads, and spell this out elementarily; then in
\cref{sec:composing} we describe the composite double category arising
from such a distributive law. \cref{sec:factorisation} gives our
zeroth example, showing how a factorisation system on a category $\C$
yields a distributive law of double categories from which $\C$ can be
recovered to within equivalence.

\cref{cocrossed} introduces co-crossed double categories and the
canonical distributive law associated to one. \cref{sec:tspans}
applies this to the horizontal Kleisli double category of a monad, so
recovering $\mathsf{T}$-spans and proving our first main theorem.
\cref{sec:calibrated_cats} then proves our second main result,
identifying co-crossed double categories of the form $\mathsf{H}J$
with B\'enabou's calibrations. In \cref{sec:calibr-bicat} we turn to
Ross's calibrated bicategories, showing that each such calibration $J$
makes an appropriately modified $\mathsf{H}J$ co-crossed, and
explaining why the converse fails. Finally,
\cref{sec:polynomial_spans} recasts the remaining theory
of~\cite{Street2020} in our setting, showing that when $\E$ is locally
cartesian closed, $\Span(\E)$ has a canonical calibration, and that
the resulting double category of corners is $\Poly(\E)$; we also
reconstruct the further examples from~\cite{Street2020}.

\textbf{Acknowledgements}. Thanks to John Bourke for helpful
discussions with the first author at Category Theory 2026.

\textbf{AI Use}. Both authors used Anthropic's Claude Opus while
preparing this paper. Most use was routine: recasting existing proofs
under changed hypotheses or conventions, regularising notation across
sections, writing \texttt{tikz-cd} diagrams and macros, drafting
routine exposition or verification from our sketches, and fixing
typos. A couple of uses went further and may be of more interest.
Firstly, the counterexample of \cref{rem:HHconverse} was proposed by
Claude, and verified by exhaustive computer search; we recast it in
the form given. (As chance would have it, it was one of Ross's
orientals.) Secondly, some open-ended discussions helped us clarify
and strengthen the point of view espoused in the introduction.

\section{Double categories}\label{sec:double-categories}

In this section we recall the (weak) double categories
of~\cite{Grandis1999Limits} and set out our notation, which generally
follows that of~\cite[Section~2]{shulman_framed}; we then present
weak double categories as pseudomonads in the weak $3$-category
$\Span(\Cat)$. It is this second description that will allow us to
speak of distributive laws between double categories in the next section.

\subsection{Elementarily}
We drop the qualifier `weak' and write simply `double category',
adding the word `strict' when horizontal composition is strictly
associative and unital.

\begin{definition}
  \label{def:double_category}
  A \emph{$\Cat$-graph} $\A$ is a span in $\Cat$:
  \[
    \begin{tikzcd}
      & \A_1 \arrow[swap]{dl}{\dom} \arrow{dr}{\cod} & \\
      \A_0 && \A_0\rlap{ .}
    \end{tikzcd}
  \]
  We refer to objects and morphisms of $\A_0$ as \emph{objects} and
  \emph{vertical arrows} of $\A$, and refer to objects and
  morphisms of $\A_1$ as \emph{horizontal arrows} and
  \emph{cells} of $\A$. We draw cells as squares
  \[
    \begin{tikzcd}
      a \arrow{r}{f} \arrow[swap]{d}{s} \cell{dr}{A} & 
      b \arrow{d}{t} \\
      c \arrow[swap]{r}{g} & d\rlap{ ,}
    \end{tikzcd}
  \]
  where here $A \colon f \to g$ in $\A_1$, $\dom A = s \colon a \to c$
  in $\A_0$ and $\cod A = t \colon b \to d$ in $\A_0$. The category
  structure of $\A_0$ provides vertical arrows with identities
  and an associative composition, while the category structure of
  $\A_1$ provides the same for the cells; we depict this latter
  composition as:
  \begin{equation*}
    \begin{tikzcd}
      \corner{a}{f}{s}{A} & 
      a' \arrow{d}{s'} \\
      \corner{b}{g}{t}{C} & 
      b' \arrow{d}{t'} \\
      c \arrow[swap]{r}{h} & c'
    \end{tikzcd} \qquad \mapsto \qquad 
    \begin{tikzcd}
      \corner{a}{f}{ts}{CA} & 
      a' \arrow{d}{t's'} \\
      c \arrow[swap]{r}{h} & c'
    \end{tikzcd}
  \end{equation*}
  where the vertical edges compose in $\A_0$.
  We refer to composition in $\A_0$ and $\A_1$ alike as
  \emph{vertical composition}, and the identities as \emph{vertical
    identities}.

  A \emph{double category} $\A$ is a $\Cat$-graph together with
  functors
  \begin{align*}
    \id \colon & \A_0 \to \A_1 & 
    * \colon & \A_1 \,{}_{\cod}\!\!\times_{\dom} \A_1 \to \A_1
  \end{align*}
  satisfying $\dom \circ \id = 1_{\A_0} = \cod \circ \id$ and
  $\dom \circ \mathord{*} = \dom \circ \pi_1$ and $\cod \circ
  \ast = \cod \circ \pi_2$, together with coherences to be specified
  shortly. This provides a composition law
  for horizontal arrows and for cells, as depicted in:
  \[
    \begin{tikzcd}
      \corner{a}{f}{s}{A} & 
      \corner{a'}{f'}{s'}{B} & 
      a'' \arrow{d}{s''} \\
      b \arrow[swap]{r}{g} & b' \arrow[swap]{r}{g'} & b''
    \end{tikzcd} \qquad \mapsto \qquad
    \begin{tikzcd}
      \corner{a}{f \ast f'}{s}{A \ast B} & 
      a'' \arrow{d}{s''} \\
      b \arrow[swap]{r}{g \ast g'} & b''\rlap{ .}
    \end{tikzcd}
  \]
  We call this \emph{horizontal composition}, with the corresponding
  identities being \emph{horizontal identities}.
  Following~\cite{shulman_framed} we write horizontal composition
  `forwards', so that the composite on the right above is $A \ast B$,
  not $B \ast A$. Horizontal composition in a double category is
  required to be associative and unital only up to specified natural
  isomorphisms $\alpha, \rho, \lambda$, whose components in $\A_1$ are
  \emph{globular cells}---cells whose vertical arrows are
  identities---of the form:
  \begin{equation*}
    \begin{tikzcd}[column sep = 4em]
      \corner{a}{f \ast (f' \ast f'')}{1}{{\alpha_{f,f',f''}}} & 
      a''' \arrow{d}{1} \\
      a \arrow[swap]{r}{(f \ast f') \ast f''} &
      a'''
    \end{tikzcd} \qquad
    \begin{tikzcd}[column sep = 4em]
      \corner{a}{f}{1}{\lambda_f} &
      a' \arrow{d}{1} \\
      a \arrow[swap]{r}{\id \ast f} &
      a'
    \end{tikzcd}
    \qquad \text{and} \qquad
    \begin{tikzcd}[column sep = 4em]
      a \idarr{d} \arrow{r}{f  \ast \id} \cell{dr}{\rho_f} &
      a'\idarr{d} \\ a \arrow[swap]{r}{f} &
      a'
    \end{tikzcd}
  \end{equation*}
  all subject to the usual axioms for a bicategory or monoidal
  category.
\end{definition}
We will, from the get-go, overload notation for vertical and
horizontal identities; in all our calculations, it will be clear
which kind of identity we are referring to. We often abbreviate a
horizontal composite $f \ast g$ to $fg$. We sometimes refer to
$\A_0$ as the \emph{category of objects} of $\A$ or the \emph{vertical
  category} of $\A$. A cell is \emph{vertically invertible} when it is
invertible in $\A_1$. The objects, horizontal arrows and globular
cells of $\A$ form a bicategory, the \emph{underlying bicategory}
of $\A$.

A \emph{double functor} $F \colon \A \to \B$ comprises functors
$F_0 \colon \A_0 \to \B_0$ and $F_1 \colon \A_1 \to \B_1$ commuting
with the domain and codomain functors, and strictly preserving
horizontal composition, horizontal identities and the coherence cells
$\alpha, \lambda, \rho$. A \emph{pseudo double functor} preserves
horizontal composition and identities only up to coherent, vertically
invertible globular cells $F(f) \ast F(g) \cong F(f \ast g)$ and
$\id_{Fa} \cong F(\id_a)$. An \emph{identity-on-objects equivalence}
$\A \simeq \B$ of double categories (the only kind we will need)
comprises identity-on-objects pseudo double functors $F \colon \A \to
\B$ and $G \colon \B \to \A$, together with vertical natural
isomorphisms $GF \cong 1$ and $FG \cong 1$ whose components are
globular cells.

\begin{example}\label{ex:sq}
  Any category $\C$ yields a strict double category $\Sq(\C)$, the
  \emph{double category of squares} in $\C$, whose vertical and
  horizontal arrows are both the arrows of $\C$, and whose cells are
  the commuting squares in $\C$.
\end{example}

\subsection{As pseudomonads}

If $\C$ is any category with pullbacks, then we have the bicategory
$\Span(\C)$, whose objects are those of $\C$, whose hom-category
$\Span(\C)(x,y)$ is the category of spans
$x \leftarrow w \rightarrow y$, and whose composition is by pullback.
A monad in $\Span(\C)$ is precisely an internal category in $\C$.

If $\C$ is a $2$-category, and its pullbacks are $2$-pullbacks, then
$\Span(\C)$ can be equipped with $3$-cells
$\alpha \colon f \Rrightarrow g$ between span maps; these involve a
$2$-cell $\alpha \colon f \Rightarrow g$ in $\C$ of the form:
\begin{equation}
  \label{eq:span_cells}
  \begin{tikzcd}[row sep = 3em]
    & \A \arrow[swap]{dl}{s} \arrow{dr}{t}
      \arrow[bend left, ""{name=R, left}]{dd}{g}
      \arrow[swap, bend right, ""{name=L, right}]{dd}{f} & \\
    \X & & \Y \\
    & \B\rlap{ ,} \arrow{ul}{u} \arrow[swap]{ur}{v} &
    \arrow[Rightarrow, from=L, to=R, "\alpha"]
  \end{tikzcd}
\end{equation}
which the legs of the spans see as identities, in the sense that
$u\alpha = 1_s$ and $v\alpha = 1_t$. With these $3$-cells, $\Span(\C)$
becomes a rather tractable kind of weak $3$-category, namely, a
bicategory in which each hom is a $2$-category, composition
is a $2$-functor, and the coherence natural transformations are
$2$-natural; in other words, a \emph{$2\text-\Cat$-enriched
  bicategory} in the sense of~\cite{Bozapalides1976Theorie}.

We may thus consider \emph{pseudomonads} in the weak $3$-category
$\Span(\C)$; these are equally well
pseudomonoids~\cite[\S 3]{Day1997Monoidal}, or in Ross's favoured
nomenclature, \emph{monoidales}, in the endohom monoidal bicategory
$\Span(\C)(X,X)$.
Spelling this out, a pseudomonad on $\X$ in $\Span(\C)$ comprises a
span $\A$ from $\X$ to $\X$, span morphisms
$\mu \colon \A \times_\X \A \to \A$ and $\eta \colon \X \to \A$, and
coherent invertible $2$-cells $\alpha, \rho, \lambda$ of the
shape~\eqref{eq:span_cells} witnessing associativity and unitality.

In particular, when $\C$ is the $2$-category $\Cat$, a pseudomonad
in $\Span(\Cat)$ is precisely a double category in the sense defined
above, with
the coherent invertible $2$-cells $\alpha, \lambda, \rho$ being
precisely the constraints of Definition~\ref{def:double_category};
note that the condition that the legs of the spans see these $2$-cells
as identities expresses precisely the \emph{globularity} of the
double-categorical constraint cells. Thus a double category is a weak
internal category in $\Cat$, just as a monad in $\Span(\C)$ is an
internal category in $\C$.

\section{Distributive laws for double categories}
\label{sec:distributive_laws}

\subsection{Between pseudomonads}
\label{subsec:between_pseudomonads}

Much as distributive laws between monads make sense in any
bicategory~\cite{Street1972The-formal}, so pseudo-distributive laws
between pseudomonads make sense in any weak
$3$-category~\cite{Marmolejo1999Distributive, cheng_pseudo-distributive_2003,
  gambino_formal_2021, walker2025}. A pseudo-distributive law relates
two pseudomonads on a single object; so if we fix this object $X$,
then the pseudo-distributive law lives entirely inside the monoidal
endohom bicategory from $X$ to $X$.

\begin{definition}
  \label{def:pseudo_distributive_law}
  Let $S$ and $T$ be pseudomonoids in a monoidal bicategory. A
  \emph{pseudo-distributive law of $S$ over $T$} comprises a $1$-cell
  $\delta \colon S \otimes T \to T \otimes S$ together with invertible
  $2$-cells
  \begin{equation*}
    \begin{tikzcd}[column sep = 3em, row sep = 3em]
      T \arrow{r}{\eta_S \otimes T}
        \arrow[""{name=cone, anchor=center, inner sep=0}, swap]{dr}{T \otimes \eta_S} &
      S \otimes T \arrow{d}{\delta} \\
      & T \otimes S
      \arrow[Rightarrow, from=cone, to=1-2, shorten <=0.6em, shorten >=0.6em, "\cone"]
    \end{tikzcd}
    \qquad\qquad
    \begin{tikzcd}[column sep = 3em, row sep = 3em]
      S \arrow{r}{S \otimes \eta_T}
        \arrow[""{name=ctwo, anchor=center, inner sep=0}, swap]{dr}{\eta_T \otimes S} &
      S \otimes T \arrow{d}{\delta} \\
      & T \otimes S
      \arrow[Rightarrow, from=1-2, to=ctwo, shorten <=0.6em, shorten >=0.6em, "\ctwo"]
    \end{tikzcd}
  \end{equation*}
  \begin{equation*}
    \begin{tikzcd}[column sep = 2.5em, row sep = 3em]
      S \otimes S \otimes T
        \arrow{r}{S \otimes \delta} \arrow[swap]{d}{\mu_S \otimes T}
        \cell{drr}{\Downarrow \cthree} &
      S \otimes T \otimes S \arrow{r}{\delta \otimes S} &
      T \otimes S \otimes S \arrow{d}{T \otimes \mu_S} \\
      S \otimes T \arrow[swap]{rr}{\delta} && T \otimes S
    \end{tikzcd}
    \qquad
    \begin{tikzcd}[column sep = 2.5em, row sep = 3em]
      S \otimes T \otimes T
        \arrow{rr}{S \otimes \mu_T} \arrow[swap]{d}{\delta \otimes T}
        \cell{drr}{\Downarrow \cfour} &&
      S \otimes T \arrow{d}{\delta} \\
      T \otimes S \otimes T \arrow[swap]{r}{T \otimes \delta} &
      T \otimes T \otimes S \arrow[swap]{r}{\mu_T \otimes S} &
      T \otimes S
    \end{tikzcd}
  \end{equation*}
  (in which we have suppressed the unit constraints of $\otimes$),
  subject to coherence axioms: nine in the original formulation
  of~\cite{Marmolejo1999Distributive}, subsequently reduced to
  five in~\cite{walker2025}.
\end{definition}
Applying this in one of the endohoms of $\Span(\Cat)$ yields the main
definition of this paper:
\begin{definition}[Distributive law between double categories]
  \label{def:dist-double}
  Let $\A$ and $\B$ be double categories with the same vertical
  category $\X_0$. A \emph{distributive law of $\A$ over $\B$} is a
  pseudo-distributive law of $\A$ over $\B$ \emph{qua} pseudomonads in
  $\Span(\Cat)(\X_0, \X_0)$.
\end{definition}

Here, in accordance with our double-categorical conventions, we drop
the ``pseudo-'' qualifier and speak simply of \emph{distributive laws}.

\subsection{Elementarily}
\label{subsec:elementarily}

In this subsection, we transcribe \cref{def:dist-double} to set out
the data and axioms of a distributive law of $\A$ over $\B$, where
$\A$ and $\B$ are double categories with common vertical category
$\X_0 = \A_0 = \B_0$:
\begin{equation*}
  \begin{tikzcd}
    & \A_1 \arrow[swap]{dl}{\dom} \arrow{dr}{\cod} & &
    & \B_1 \arrow[swap]{dl}{\dom} \arrow{dr}{\cod} & \\
    \X_0 && \X_0 & \X_0 && \X_0\rlap{ .}
  \end{tikzcd}
\end{equation*}
The first datum for a distributive law is
a functor
\begin{equation}
  \label{eq:deltafunctor}
  \delta \colon \A_1 \,{}_{\cod}\!\!\times_{\dom} \B_1 \longrightarrow
  \B_1 \,{}_{\cod}\!\!\times_{\dom} \A_1\rlap{ .}
\end{equation}
Objects of the domain category $\A_1 \,{}_{\cod}\!\!\times_{\dom} \B_1$
are pairs $(f,u)$, where $f$ and $u$ are horizontal arrows of $\A$,
respectively $\B$, and $\cod(f) = \dom(u)$ in $\X_0$; we display this
to the left in: 
\[\begin{tikzcd}
    a \harrow{r}{f} & b \karrow{r}{u} & c
\end{tikzcd} \qquad\qquad
\begin{tikzcd}
    \hcorner{a}{f}{s}{\sigma} & 
    \kcorner{b}{u}{r}{\tau} & 
    c \arrow{d}{t} \\
    a' \harrows{r}{f'} & 
    b' \karrows{r}{u'} &
    c'\rlap{ .}
  \end{tikzcd}\] Similarly, arrows of
$\A_1 \,{}_{\cod}\!\!\times_{\dom} \B_1$ are pairs of cells from $\A$
and $\B$ that abut along a vertical arrow, as seen to the right above.
We call such pairs of horizontal arrows or cells \emph{formally
  composable}.

The codomain of $\delta$ involves pairs which are formally composable
in the \emph{other} order ($\B$ then $\A$). Thus, $\delta$'s action on
objects and arrows is an assignment:
\begin{equation*}
  \label{eq:dist-assignment}\tag{$\delta$}
  \begin{tikzcd}
    \hcorner{a}{f}{s}{\sigma} & \kcorner{b}{u}{}{\tau} & c \arrow{d}{t} \\
    a' \harrows{r}{f'} & b' \karrows{r}{u'} & c'
  \end{tikzcd}
  \qquad
  \mapsto
  \qquad
  \begin{tikzcd}
    \kcorner{a}{u_f}{s}{\tau_\sigma} &
    \hcorner{\cdot}{f_u}{}{\sigma_\tau} &
    c \arrow{d}{t} \\
    a' \karrows{r}{u'_{f'}} &
    \cdot \harrows{r}{f'_{u'}} &
    c'\rlap{ .}
  \end{tikzcd}
\end{equation*}
Note that the middle vertical arrow on the right is a datum chosen by
the functor; it is $\cod(\tau_\sigma) = \dom(\sigma_\tau)$ and we
leave it unlabelled in all our diagrams. The fact that $\delta$ is a
map of spans over $\X_0$ is visible in the fact that the outer
vertical boundaries $(s,t)$ are unchanged by the assignment. On
objects, $\delta$ is the double-categorical form of the
swap~\eqref{eq:swap} of the introduction, now acting on cells as well
as arrows.

We now describe the coherence isomorphisms $\cone, \dots,
\cfour$. The first of these is an invertible natural transformation
\begin{equation*}
  \begin{tikzcd}[column sep = 4em, row sep = 3em]
    \B_1 \arrow{r}{\id \,{}_{\cod}\!\times_{\dom} \B_1}
    \arrow[""{name=cone, anchor=center, inner sep=0}, swap]{dr}{\B_1 \,{}_{\cod}\!\times_{\dom} \id} &
    \A_1 \,{}_{\cod}\!\!\times_{\dom} \B_1 \arrow{d}{\delta} \\
    & \B_1 \,{}_{\cod}\!\!\times_{\dom} \A_1
    \arrow[Rightarrow, from=cone, to=1-2, shorten <=0.6em, shorten >=0.6em, "\cone"]
  \end{tikzcd}
\end{equation*}
which the span projections see as an identity. This amounts to giving:
\begin{itemize}
\item For each horizontal arrow
  \[\begin{tikzcd}
      a \karrow{r}{u} & b
    \end{tikzcd}\] of $\B$, a pair of vertically invertible, formally composable cells
  \begin{equation*}
    \label{eq:cone}\tag{$\cone$}
    \begin{tikzcd}
      a \karrow{r}{u} \idarr{d} \arrow[cell label, \kcolor]{dr}{\cone} & b \arrow[cell label, \hcolor]{dr}{\cone} \arrow{d} \ar[\hcolor]{r}{\id} & b \idarr{d}\\
      a \karrows{r}{u_{\id}} & \cdot \harrows{r}{\id_u} & b\rlap{ ,}
    \end{tikzcd}
  \end{equation*}
  naturally with respect to vertical composition.
\end{itemize}
We commit a small abuse of notation by naming both of these cells
$\omega_1$; this is harmless since they live in different double
categories. Note that the fact that the span projections see the
natural transformation as an identity is visible in the fact that the
outer vertical boundaries of~\eqref{eq:cone} are identity arrows. By
contrast, the middle vertical arrow is not so constrained: it is part
of the data of $\omega_1$, determined as the vertical arrow along
which the two cells abut. This arrow satisfies its own constraint: its
codomain must be the unnamed object provided
by~\eqref{eq:dist-assignment}.

It is now unproblematic to spell out the data for $\ctwo, \cthree$ and
$\cfour$ in the same manner:

\begin{itemize}
\item For each horizontal arrow
  \[\begin{tikzcd}
      a \harrow{r}{f} & b
    \end{tikzcd}\]
  of $\A$, a pair of vertically invertible, formally composable cells
  \begin{equation*}
    \label{eq:ctwo}\tag{$\ctwo$}
    \begin{tikzcd}
      a \karrow{r}{\id_{f}} \idarr{d} \arrow[cell label, \kcolor]{dr}{\ctwo} & \cdot \arrow[cell label, \hcolor]{dr}{\ctwo} \arrow{d} \ar[\hcolor]{r}{f_{\id}} & b \idarr{d}\\
      a \karrows{r}{\id} & a \harrows{r}{f} & b\rlap{ ,}
    \end{tikzcd}
  \end{equation*}
  naturally with respect to vertical composition.
\item For each diagram
  \[\begin{tikzcd}
      a \harrow{r}{f} & b \harrow{r}{g} & c \karrow{r}{u} & d 
    \end{tikzcd}
  \]
  a pair of vertically invertible, formally composable cells
  \begin{equation*}
    \label{eq:cthree}\tag{$\cthree$}
    \begin{tikzcd}
      a \karrow{r}{(u_g)_f} \idarr{d} \kcell{dr}{\cthree} &
      \cdot \harrow{r}{f_{u_g}} \ar[d] \hcell{drr}{\cthree} & 
      \cdot \harrow{r}{g_u} & 
      d \idarr{d} \\
      a \karrows{r}{u_{fg}} &
      \cdot \harrows{rr}{(fg)_u } && 
      d\rlap{ ,}
    \end{tikzcd}
    \end{equation*}
  naturally with respect to vertical composition;
\item For each diagram
  \[\begin{tikzcd}
      a \harrow{r}{f} & b \karrow{r}{u} & c \karrow{r}{v} & d 
    \end{tikzcd}\]
  a pair of vertically invertible, formally composable cells
  \begin{equation*}
    \label{eq:cfour}\tag{$\cfour$}
    \begin{tikzcd}
      a \karrow{rr}{(uv)_{f}} \idarr{d} \kcell{drr}{\cfour} &&
      \cdot \harrow{r}{f_{uv}} \ar[d] \hcell{dr}{\cfour} 
      & d \idarr{d}\\
      a \karrows{r}{u_f} & \cdot \karrows{r}{v_{f_u}} & \cdot \harrows{r}{(f_u)_v} & d\rlap{ ,}
    \end{tikzcd}
  \end{equation*}
  naturally with respect to vertical composition.
\end{itemize}

Note that in each of $\cone$--$\cfour$, the middle vertical arrow is
part of the data; if, as is often the case, cells in $\A$ and $\B$ are
determined by their boundaries, it is in fact \emph{all} of the data.

Finally, we spell out the five coherence conditions these natural
transformations must satisfy. Our conventions align with those
of~\cite{walker2025}, and so we can read off the shape of these axioms
from \emph{loc.~cit.}: we need only take Walker's $T$ and $P$ to be
our $\A$ and $\B$. Evaluating at components then yields the following
five conditions. Note that each expresses an equality of cells in
$\B_1 \,{}_{\cod}\!\!\times_{\dom} \A_1$, which is equally well a pair
of equalities, one in $\B_1$ and one in $\A_1$.
\begin{itemize}
\item For each diagram
  \[\begin{tikzcd}
      a \harrow{r}{f} & b \karrow{r}{u} & c 
    \end{tikzcd}\]
  we have the equality:
  \begin{equation*}
    \tag{\MakeUppercase{\wone}}
    \label{eq:wone}
    \begin{tikzcd}[column sep = 3.5em]
      a \karrow{r}{{u}_{f}} \idarr{d} \kcell{dr}{(\cone)_f} &
      \cdot \harrow{r}{{f}_{u}} \idarr{d} \hcell{dr}{f_{\cone}} &
      c \harrow{r}{\id} \ar[d] \hcell{dr}{\cone} & 
      c \idarr{d} \\
      a \karrowm{r}{{(u_\id)}_{f}} \idarr{d} \kcell{dr}{\cthree} &
      \cdot \harrowm{r}{{f}_{(u_\id)}} \ar[d] \hcell{drr}{\cthree} &
      c \harrowm{r}{\id_u} & 
      c \idarr{d} \\
      a \karrowm{r}{u_{f \ast \id}} \idarr{d} \kcell{dr}{u_{\rho}} &
      \cdot \harrowm{rr}{(f \ast \id)_u} \idarr{d}  \hcell{drr}{\rho_u} &&
      c \idarr{d} \\
      a \karrows{r}{u_f} &
      \cdot \harrows{rr}{f_u} &&
      c
    \end{tikzcd}
    \ \ =\ \ 
    \begin{tikzcd}
      a \karrow{r}{{u}_{f}} \idarr{d} &
      \cdot \harrow{r}{{f}_{u}} \idarr{d} \hcell{drr}{\rho} &
      c \harrow{r}{\id} & 
      c \idarr{d} \\
      a \karrows{r}{u_f} &
      \cdot \harrows{rr}{f_u} &&
      c\rlap{ .}
    \end{tikzcd}
  \end{equation*}
\item For each diagram
  \[\begin{tikzcd}
      a \harrow{r}{f} & b \karrow{r}{u} & c 
    \end{tikzcd}\]
  we have the equality
  \begin{equation*}
    \tag{\MakeUppercase{\wtwo}}
    \label{eq:wtwo}
    \begin{tikzcd}[column sep = 3em]
      a \karrow{rr}{{u}_{f}} \idarr{d} \kcell{drr}{\lambda_f} &&
      \cdot \harrow{r}{{f}_{u}} \idarr{d} \hcell{dr}{f_\lambda} & 
      c \idarr{d} \\
      a \karrowm{rr}{{(\id \ast u)}_{f}} \idarr{d} \kcell{drr}{\cfour} &&
      \cdot \harrowm{r}{{f}_{\id \ast u}} \ar[d] \hcell{dr}{\cfour} & 
      c \idarr{d} \\
      a \karrowm{r}{{\id}_{f}} \idarr{d} \kcell{dr}{\ctwo} &
      \cdot \karrowm{r}{{u}_{f_{\id}}} \idarr{d} \kcell{dr}{u_{\ctwo}} &
      \cdot \harrowm{r}{{(f_\id)}_{u}} \ar[d] \hcell{dr}{(\ctwo)_u} & 
      c \idarr{d} \\
      a \karrows{r}{\id} &
      a \karrows{r}{{u}_{f}} &
      \cdot \harrows{r}{{f}_{u}} & 
      c
    \end{tikzcd}
    \ \ = \ \ 
    \begin{tikzcd}
      a \karrow{rr}{{u}_{f}} \idarr{d} \kcell{drr}{\lambda} &&
      \cdot \harrow{r}{{f}_{u}} \idarr{d} & 
      c \idarr{d} \\
      a \karrows{r}{\id} &
      a \karrows{r}{{u}_{f}} &
      \cdot \harrows{r}{{f}_{u}} & 
      c\rlap{ .}
    \end{tikzcd}
  \end{equation*}
\item For each diagram
  \[\begin{tikzcd}
      a \harrow{r}{f} & b \harrow{r}{g} & c \harrow{r}{h} & d \karrow{r}{u} & e
    \end{tikzcd}\]
  we have the equality
  \begin{equation*}
    \tag{\MakeUppercase{\wthree}}
    \label{eq:wthree}
    \begin{tikzcd}[column sep = 2em]
      a \karrow{r}{((u_h)_g)_f} \idarr{d} \kcell{dr}{(\cthree)_f} &[1.5em]
      \cdot \harrow{r}{f_{(u_h)_g}} \ar[d] \hcell{dr}{f_{\cthree}} &[1em]
      \cdot \harrow{r}{g_{u_h}} \hcell{drr}{\cthree} \ar[d] &
      \cdot \harrow{r}{h_u} &
      e \idarr{d} \\
      a \karrowm{r}{(u_{gh})_f} \idarr{d} \kcell{dr}{\cthree} &
      \cdot \harrowm{r}{f_{u_{gh}}} \ar[d] \hcell{drrr}{\cthree} &
      \cdot \harrowm{rr}{(gh)_u} &&
      e \idarr{d} \\
      a \karrowm{r}{u_{f(gh)}} \idarr{d} \kcell{dr}{u_\alpha} &
      \cdot \harrowm{rrr}{(f(gh))_u} \idarr{d} \hcell{drrr}{\alpha_u} &&&
      e \idarr{d} \\
      a \karrows{r}{u_{(fg)h}} &
      \cdot \harrows{rrr}{((fg)h)_u} &&&
      e
    \end{tikzcd}
    \ =\ 
    \begin{tikzcd}[column sep = 2em]
      a \karrow{r}{((u_h)_g)_f} \idarr{d} \kcell{dr}{\cthree} &[1.5em]
      \cdot \harrow{r}{f_{(u_h)_g}} \ar[d] \hcell{drr}{\cthree} &[1em]
      \cdot \harrow{r}{g_{u_h}} &
      \cdot \harrow{r}{h_u} \idarr{d} &
      e \idarr{d} \\
      a \karrowm{r}{(u_h)_{fg}} \idarr{d} \kcell{dr}{\cthree} &
      \cdot \harrowm{rr}{(fg)_{u_h}} \ar[d] \hcell{drrr}{\cthree} &&
      \cdot \harrowm{r}{h_u} &
      e \idarr{d} \\
      a \karrows{r}{u_{(fg)h}} &
      \cdot \harrows{rrr}{((fg)h)_u} &&&
      e
    \end{tikzcd}
  \end{equation*}
\item For each diagram
  \[\begin{tikzcd}
      a \harrow{r}{f} & b \karrow{r}{u} & c \karrow{r}{v} & d \karrow{r}{w} & e
    \end{tikzcd}\]
  we have the equality
  \begin{equation*}
    \tag{\MakeUppercase{\wfour}}
    \label{eq:wfour}
    \begin{tikzcd}[column sep = 2em]
      a \karrow{rrr}{{(u(vw))}_{f}} \idarr{d} \kcell{drrr}{\cfour} &&&[1em]
      \cdot \harrow{r}{{f}_{u(vw)}} \ar[d] \hcell{dr}{\cfour} &[1.5em]
      e \idarr{d} \\
      a \karrowm{r}{u_{f}} \idarr{d} &
      \cdot \karrowm{rr}{(vw)_{f_{u}}} \idarr{d} \kcell{drr}{\cfour} &&
      \cdot \harrowm{r}{{(f_u)}_{vw}} \ar[d] \hcell{dr}{\cfour} &
      e \idarr{d} \\
      a \karrows{r}{u_f} &
      \cdot \karrows{r}{{v}_{f_u}} &
      \cdot \karrows{r}{{w}_{(f_u)_v}} &
      \cdot \harrows{r}{((f_u)_v)_w} & 
      e
    \end{tikzcd}
    \ =\ 
    \begin{tikzcd}[column sep = 2em]
      a \karrow{rrr}{{(u(vw))}_{f}} \idarr{d} \kcell{drrr}{\alpha_f} &&&[1em]
      \cdot \harrow{r}{{f}_{u(vw)}} \ar[d] \hcell{dr}{f_\alpha} &[1.5em]
      e \idarr{d} \\
      a \karrowm{rrr}{{((uv)w)}_{f}} \idarr{d} \kcell{drrr}{\cfour} &&&
      \cdot \harrowm{r}{{f}_{(uv)w}} \ar[d] \hcell{dr}{\cfour} &
      e \idarr{d} \\
      a \karrowm{rr}{(uv)_{f}} \idarr{d} \kcell{drr}{\cfour} &&
      \cdot \karrowm{r}{w_{f_{uv}}} \ar[d] \kcell{dr}{w_{\cfour}} &
      \cdot \harrowm{r}{{(f_{uv})}_{w}} \ar[d] \hcell{dr}{(\cfour)_w} &
      e \idarr{d} \\
      a \karrows{r}{u_f} &
      \cdot \karrows{r}{{v}_{f_u}}  &
      \cdot \karrows{r}{{w}_{(f_u)_v}}  &
      \cdot \harrows{r}{((f_u)_v)_w}  & 
      e 
    \end{tikzcd}
  \end{equation*}
\item For each diagram
  \[\begin{tikzcd}
      a \harrow{r}{f} & b \harrow{r}{g} & c \karrow{r}{u} & d \karrow{r}{v} & e
    \end{tikzcd}\]
  we have the equality
  \begin{equation*}
    \tag{\MakeUppercase{\wfive}}
    \label{eq:wfive}
    \begin{tikzcd}[column sep = 1.5em]
      a \karrow{rr}{((uv)_g)_f} \idarr{d} \kcell{drr}{\cthree} &&
      \cdot \harrow{r}{f_{(uv)_g}} \ar[d] \hcell{drr}{\cthree} &
      \cdot \harrow{r}{g_{uv}} &
      e \idarr{d} \\
      a \karrowm{rr}{(uv)_{fg}} \idarr{d} \kcell{drr}{\cfour} &&
      \cdot \harrowm{rr}{(fg)_{uv}} \ar[d] \hcell{drr}{\cfour} &&
      e \idarr{d} \\
      a \karrows{r}{u_{fg}} &
      \cdot \karrows{r}{v_{(fg)_u}} &
      \cdot \harrows{rr}{((fg)_u)_v} &&
      e
    \end{tikzcd}
    \ =\ 
    \begin{tikzcd}[column sep = 3em]
      a \karrow{rr}{((uv)_g)_f} \idarr{d} \kcell{drr}{(\cfour)_f} &&[1.5em]
      \cdot \harrow{r}{f_{(uv)_g}} \ar[d] \hcell{dr}{f_{\cfour}} &[1.5em]
      \cdot \harrow{r}{g_{uv}} \ar[d] \hcell{dr}{\cfour} &
      e \idarr{d} \\
      a \karrowm{rr}{(u_g \ast v_{g_u})_{f}} \idarr{d} \kcell{drr}{\cfour} &&
      \cdot \harrowm{r}{f_{(u_g \ast v_{g_u})}} \ar[d] \hcell{dr}{\cfour} &
      \cdot \harrowm{r}{(g_u)_v} \idarr{d} &
      e \idarr{d} \\
      a \karrowm{r}{(u_g)_f} \idarr{d} &
      \cdot \karrowm{r}{(\!v_{g_u}\!)_{(\!f_{u_g}\!)}} \idarr{d} \kcell{dr}{\cthree} &
      \cdot \harrowm{r}{(\!f_{u_g}\!)_{(\!v_{g_u}\!)}} \ar[d] \hcell{drr}{\cthree} &
      \cdot \harrowm{r}{(g_u)_v} &
      e \idarr{d} \\
      a \karrowm{r}{(u_g)_f} \idarr{d} \kcell{dr}{\cthree} &
      \cdot \karrowm{r}{v_{(f_{u_g} \ast g_u)}} \ar[d] \kcell{dr}{v_{\cthree}} &
      \cdot \harrowm{rr}{{(f_{u_g} \ast g_u)}_v} \ar[d] \hcell{drr}{(\cthree)_v} &&
      e \idarr{d} \\
      a \karrows{r}{u_{fg}} &
      \cdot \karrows{r}{v_{(fg)_u}} &
      \cdot \harrows{rr}{((fg)_u)_v} &&
      e
    \end{tikzcd}
  \end{equation*}
\end{itemize}

\section{Composing double categories}\label{sec:composing}

A pseudo-distributive law of one pseudomonad over another induces a
composite pseudomonad, as in~\cite[\S 5]{Marmolejo1999Distributive}.
In particular, each distributive law of double categories, $\A$
over $\B$, determines a composite double category, which we denote by
$\B \bt \A$, and which we now describe explicitly.

First of all, the
underlying span of $\B \bt \A$ is the pullback:
\[\begin{tikzcd}[sep = {3.5em, between origins}]
    && \B_1 \,{}_{\cod}\!\!\times_{\dom} \A_1 \ar[dl] \ar[dr] \pb{dd} && \\
    & \B_1 \arrow[swap]{dl}{\dom} \arrow{dr}{\cod} & & \A_1 \arrow[swap]{dl}{\dom} \arrow{dr}{\cod} & \\
    \X_0 && \X_0 && \X_0\rlap{ ;}
  \end{tikzcd}\]
thus, objects and vertical arrows are as in $\A$ and $\B$, while
horizontal arrows and cells are formally composable pairs in the order
$\B$-then-$\A$. Vertical composition is componentwise.

Horizontal composition of arrows and cells is given by:
\begin{equation}
  \label{eq:hcomp-bowtie}
  \begin{tikzcd}[ampersand replacement=\&, sep={3.5em, between origins}]
    \kcorner{a}{u}{s}{\tau} \& 
    \hcorner{b}{f}{}{\sigma} \&
    c \arrow{d}{t} \cell{dr}{\ast} \&[1.5em]
    \kcorner{c}{v}{t}{\rho} \& 
    \hcorner{d}{g}{}{\pi} \& 
    e \arrow{d}{r} \cell{dr}{=} \&[2em]
    \kcorner{a}{u}{s}{\tau} \& 
    \kcorner{b}{v_f}{}{\rho_\sigma} \& 
    \hcorner{\cdot}{f_v}{}{\sigma_\rho} \& 
    \hcorner{d}{g}{}{\pi} \& 
    e \arrow{d}{r} \\
    a' \karrows{r}{u'} \&
    b' \harrows{r}{f'} \& 
    c' \&
    c' \karrows{r}{v'} \&
    d' \harrows{r}{g'} \& 
    e' \&
    a' \karrows{r}{u'} \&
    b' \karrows{r}{v'_{f'}} \&
    {\cdot} \harrows{r}{f'_{v'}} \& 
    d' \harrows{r}{g'} \& 
    e'\rlap{ ,}
  \end{tikzcd}
\end{equation}
where we interchange the middle two cells via $\delta$. When $\delta$
is given by pullback, as in \cref{cor:spans}, this is composition of
spans by pullback. Finally, we describe the three coherence
isomorphisms $\alpha$, $\lambda$ and $\rho$ for $\B \bt \A$:
\begin{itemize}
    \item 
The component of $\alpha$ at a composable triple of maps
\[
\begin{tikzcd}
x_0 \karrow{r}{u} & 
x_1 \harrow{r}{f} & 
x_2 \karrow{r}{v} & 
x_3 \harrow{r}{g} & 
x_4 \karrow{r}{w} & 
x_5 \harrow{r}{h} & 
x_6
\end{tikzcd}
\]
is given by the following formally composable pair of invertible
cells, in which we suppress the associativity constraints of $\A$
and $\B$:
\[
\begin{tikzcd}[column sep = 4em]
x_0 \karrow{r}{u} \idarr{d} & 
x_1 \karrow{rr}{(v w_g)_f} \idarr{d} \kcell{drr}{\cfour} &&
\cdot \harrow{r}{f_{v w_g}} \ar[d] \hcell{dr}{\cfour}& 
\cdot \harrow{r}{g_w} \idarr{d} & 
x_5 \harrow{r}{h} \idarr{d} & 
x_6 \idarr{d} \\
x_0 \karrowm{r}{u} \idarr{d} & 
x_1 \karrowm{r}{v_f} \idarr{d} & 
\cdot \karrowm{r}{(w_g)_{f_v}} \idarr{d} \kcell{dr}{\cthree} & 
\cdot \harrowm{r}{(f_v)_{{w_g}}} \ar[d] \hcell{drr}{\cthree} &
\cdot \harrowm{r}{g_w} & 
x_5 \harrow{r}{h} \idarr{d} & 
x_6 \idarr{d} \\
x_0 \karrows{r}{u} & 
x_1 \karrows{r}{v_f} & 
\cdot \karrows{r}{w_{f_v g}} & 
\cdot \harrows{rr}{(f_v g)_w} && 
x_5 \harrows{r}{h} & 
x_6
\end{tikzcd}
\]
\item The components of $\lambda$ and $\rho$ at an arrow
\[\begin{tikzcd}
    a \karrow{r}{u} & 
    b \harrow{r}{f} & 
    c
  \end{tikzcd}\] of $\B \bt \A$ are given by the
formally composable pairs of invertible cells:
\[\begin{tikzcd}
    a \karrow{rr}{u} \idarr{d} \kcell{drr}{\lambda} &&
    b \harrow{rr}{f} \idarr{d} \hcell{drr}{\lambda} &&
    c \idarr{d} \\
    a \karrowm{r}{\id} \idarr{d} & 
    a \karrowm{r}{u} \idarr{d} \kcell{dr}{\cone} & 
    b \harrowm{r}{\id} \ar[d] \hcell{dr}{\cone} & 
    b \harrowm{r}{f} \idarr{d} &
    c \idarr{d} \\
    a \karrows{r}{\id} & 
    a \karrows{r}{u_\id} & 
    \cdot \harrows{r}{\id_u} &
    b \harrows{r}{f} &
    c
\end{tikzcd}\]
and
\[\begin{tikzcd}
    a \karrow{r}{u} \idarr{d} & 
    b \karrow{r}{\id_f} \idarr{d} \kcell{dr}{\ctwo} & 
    \cdot \harrow{r}{f_\id} \ar[d] \hcell{dr}{\ctwo} & 
    c \harrow{r}{\id} \idarr{d} &
    c \idarr{d} \\
    a \karrowm{r}{u} \idarr{d} \kcell{drr}{\rho} & 
    b \karrowm{r}{\id} & 
    b \harrowm{r}{f} \idarr{d} \hcell{drr}{\rho} & 
    c \harrowm{r}{\id} &
    c \idarr{d} \\
    a \karrows{rr}{u} &&
    b \harrows{rr}{f} &&
    c\rlap{ .}
\end{tikzcd}\]
\end{itemize}
The fact that these data satisfy the unit and pentagon axioms is a
consequence of the general theory of pseudo-distributive laws, but can
also be shown directly. To check that the associativity pentagon
commutes one will need to use the corresponding axioms for $\A$ and
$\B$, axioms \ref{eq:wthree}, \ref{eq:wfour}, \ref{eq:wfive}, and
naturality of all coherence data. For the unit axiom, one will need to
use the corresponding axioms for $\A$ and $\B$, axioms \ref{eq:wone},
\ref{eq:wtwo}, and naturality of all coherence data.

\section{Zeroth example: factorisation systems}\label{sec:factorisation}

In this section we show that a factorisation system on a category $\C$
is captured by a distributive law between two double categories, for
which the composite double category is, to within equivalence, $\C$.
We begin with the strict case of~\cite{rosebrugh_distributive_2002}.

Double categories $\A$ which are vertically discrete---meaning that
the categories $\A_0$ and $\A_1$ are discrete categories---are just
categories, and a distributive law between two such is a distributive
law between categories as in~\cite{rosebrugh_distributive_2002}. In
this case, the composite double category is again discrete, and
coincides with the composite category of \emph{ibid}. So our
notion generalises Rosebrugh and Wood's.

A key point in~\cite{rosebrugh_distributive_2002} is that the
composite monad of a distributive law can be decomposed into its
constituents by way of an extra structure borne by the composite
category $\C$: that of a \emph{strict factorisation system}.
\begin{definition}
  A \emph{strict factorisation system} on a category $\C$ comprises
  subcategories $\E$ and $\M$ of $\C$, each containing every object,
  such that each map $f$ of $\C$ factorises uniquely as
  $f = me$ with $e \in \E$ and $m \in \M$.
\end{definition}

Given a strict factorisation system on a category, we recover a
distributive law between the constituents $\E$ and $\M$ whose
composite monad is, up to labelling of morphisms, precisely $\C$. In
this sense, distributive laws are the same as strict factorisation
systems.

Now, strict factorisation systems do arise in nature---for example,
the (degeneracy, face) strict factorisation system on $\Delta$---but
more common in practice are the unqualified \emph{factorisation
  systems}:

\begin{definition}
  A \emph{factorisation system} on a category $\C$ comprises
  subcategories $\E$ and $\M$ of $\C$, each containing every object,
  such that:
  \begin{enumerate}[label=(\roman*),itemsep=0em]
  \item $\E$ and $\M$ each contain every isomorphism of $\C$;
  \item every map $f$ of $\C$ factorises as $f = me$ with $e \in \E$
    and $m \in \M$;
  \item each $e \in \E$ is orthogonal to each $m \in \M$, in the sense
    that for every $e \in \E$ and $m \in \M$ fitting into a commuting
    square as below, there is a unique $d \colon x \to y$ with
    $de = u$ and $md = v$.
    \begin{equation*}
      \begin{tikzcd}
        w \arrow{r}{u} \arrow[swap]{d}{e} &
        y \arrow{d}{m} \\
        x \arrow[swap]{r}{v} \arrow[dashed]{ur}{d} &
        z\rlap{ .}
      \end{tikzcd}
    \end{equation*}
  \end{enumerate}  
\end{definition}

The nomenclature is unfortunate: a strict factorisation system is not
in general a factorisation system. However, if we close its two
classes of maps under composition with isomorphisms in $\C$, then it
becomes one, which partially excuses the adjectival form.

In their \S 5.1, Rosebrugh and Wood ask whether non-strict
factorisation systems can be captured by replacing the equations of a
distributive law with coherent isomorphisms; as we recalled in the
introduction, they note that they lack an appropriate dimension in
which to do so. Their own remedy is to pass to the bicategory of
profunctors between groupoids, wherein a factorisation system $(\E,
\M)$ on $\C$ \emph{is} a distributive law, between $\E$ and $\M$ qua
monads on the groupoid of isomorphisms of $\C$. In our setting, the
isomorphisms live instead in the vertical direction of a double
category.

\begin{definition}
  Let $\C$ be a category and $\A$ a wide subcategory of $\C$. We write
  ${\Sq(\C)}|_{\A}$ for the sub-double category of $\Sq(\C)$
  containing all the objects and vertical maps, but only
  the $\A$-maps horizontally, along with all squares between them.
\end{definition}

\begin{proposition}\label{prop:fs-distlaw}
  A factorisation system $(\E,\M)$ on a category $\C$ gives rise to a
  distributive law $\delta$ of the double category $\Sq(\C)|_{\M}$ over
  $\Sq(\C)|_{\E}$.
\end{proposition}
\begin{proof}
  On objects, $\delta$ acts via
  \begin{equation*}
    \begin{tikzcd}
      a \harrow{r}{m \in \M} & b \karrow{r}{e \in \E} & c
    \end{tikzcd} \qquad \mapsto \qquad 
    \begin{tikzcd}
      a \karrow{r}{e_m \in \E} & \cdot \harrow{r}{m_e \in \M} & c
    \end{tikzcd} 
  \end{equation*}
  where $m_e \circ e_m$ is a chosen $(\E, \M)$-factorisation of the
  composite $em$. Before giving its action on morphisms, we observe
  that, for any commuting diagram as around the outside of
  \begin{equation*}\tag{*}\label{eq:fs-filler}
    \begin{tikzcd}
      \kcorner{a}{e_1}{s}{} & \cdot \harrow{r}{m_1} \arrow[dashed]{d}{j} &
      c \arrow{d}{t} \\
      a' \karrows{r}{e_2} & \cdot \harrows{r}{m_2} & c'\rlap{ ,}
    \end{tikzcd}
  \end{equation*}
  there is, by orthogonality of $e_1$ and $m_2$, a unique $j$ as
  displayed making both squares commute, which is invertible whenever
  $s$ and $t$ are so. Using this, the action on
  morphisms of $\delta$ is now given by the assignment
  \begin{equation*}
    \begin{tikzcd}
      \hcorner{a}{m}{s}{=} & \kcorner{b}{e}{}{=} & c \arrow{d}{t} \\
      a' \harrows{r}{m'} & b' \karrows{r}{e'} & c'
    \end{tikzcd}
    \qquad
    \mapsto
    \qquad
    \begin{tikzcd}
      \kcorner{a}{e_m}{s}{=} &
      \hcorner{\cdot}{m_e}{j}{=} &
      c \arrow{d}{t} \\
      a' \karrows{r}{e'_{m'}} &
      \cdot \harrows{r}{m'_{e'}} &
      c'\rlap{ ,}
    \end{tikzcd}
  \end{equation*}
  wherein $j$ is the unique morphism assured by~\eqref{eq:fs-filler}.

  For the coherence data $\cone$--$\cfour$, observe that in each
  case, the required invertible cells inhabit a commuting diagram of
  the shape~\eqref{eq:fs-filler}, with both sides invertible, and the
  only missing datum being the interior vertical map. Thus, existence
  in~\eqref{eq:fs-filler} yields
  $\cone$--$\cfour$. Finally, for functoriality of $\delta$,
  naturality of $\cone$--$\cfour$, and the axioms
  (\ref{eq:wone})--(\ref{eq:wfive}), we observe that cells of
  $\Sq(\C)$ are determined by their boundaries, so that, as noted
  after~\eqref{eq:cfour}, each equality amounts to equality of the
  interior vertical maps on its two sides. Since each square commutes,
  this follows by uniqueness in~\eqref{eq:fs-filler}.
\end{proof}

In the situation of the previous result, it would be tempting to think
that the composite double category
$\Sq(\C)|_\E \bt \Sq(\C)|_\M$ would be $\Sq(\C)$. This is not the
case: for indeed, horizontal arrows from $a$ to $c$ in the composite
double category are formal composites
\begin{equation*}
 \begin{tikzcd}
   a \karrow{r}{e \in \E} & \cdot \harrow{r}{m \in \M} & c\rlap{ ,}
 \end{tikzcd} 
\end{equation*}
each of which certainly yields an arrow of $\C$ via composition---but
not \emph{uniquely} so, since an arrow of $\C$ may have many distinct
$(\E, \M)$-factorisations. Instead, we have an identity-on-objects
equivalence of double categories:
\begin{equation*}
  \begin{tikzcd}[column sep = 4em]
    \Sq(\C)|_\E \bt \Sq(\C)|_\M
      \arrow[shift left=1ex]{r}{R} &
    \Sq(\C)\rlap{ ,}
      \arrow[shift left=1ex]{l}{L}
  \end{tikzcd}
\end{equation*}
where $R$ and $L$ act as the identity on objects
and vertical arrows, and act on horizontal morphisms and cells by
composition and by factorisation respectively. Indeed, we have that
$RL$ is the identity double functor, while $LR$ is isomorphic to the
identity via the invertible globular cells $(e, m) \to LR(e, m)$
provided by~\eqref{eq:fs-filler}. Note that $R$ is a double functor, but $L$
is only a pseudo double functor; its comparison cells $L(f) \ast L(g) \cong L(f \ast g)$
are once again the globular isomorphisms provided
by~\eqref{eq:fs-filler}.

In particular, the underlying bicategory of
$\Sq(\C)|_\E \bt \Sq(\C)|_\M$ is equivalent (in the $2$-category of
bicategories, pseudofunctors and icons~\cite{Lack2010Icons}) to the
underlying bicategory of $\Sq(\C)$, which is $\C$ itself, seen as a
locally discrete bicategory.

\begin{remark}\label{rem:stepan}
  \v{S}t\v{e}p\'an~\cite{stepan2023} also describes factorisation
  systems by way of double categories, in a manner related to, but in
  a sense orthogonal to, ours. Given a factorisation system $(\E, \M)$
  on $\C$, he considers a \emph{single} double category $D_{\E,\M}$,
  the locally full sub-double category of $\Sq(\C)$ whose vertical and
  horizontal arrows are the $\E$- and $\M$-maps, and re-finds $\C$
  from it via the \emph{category of corners} construction
  of~\cite{weber2015}; in this way, \v{S}t\v{e}p\'an obtains an
  equivalence between the category of factorisation systems and a
  suitable category of \emph{factorisation double categories}.

  Thus, \v{S}t\v{e}p\'an places $\E$ and $\M$ in the two different
  directions of a single double category; while we place $\E$ and $\M$
  in the same, horizontal, direction of two distinct double
  categories. These apparently distinct approaches can be reconciled
  by realising the double categories in question as the faces of a
  single \emph{triple category}. Its objects are those of $\C$; its
  arrows in the three directions are the $\E$-maps, the $\M$-maps, and
  all maps of $\C$; and its faces and cubes are commuting squares and
  cubes in $\C$.

  The face spanned by the first two directions is \v{S}t\v{e}p\'an's
  $D_{\E,\M}$, while those spanned by the other two pairs of
  directions are our $\Sq(\C)|_\E$ and $\Sq(\C)|_\M$. Our distributive
  law involves all three faces, in much the same way that the
  distributive law~\eqref{eq:swap} of the introduction involves both
  directions of $D_{\E,\M}$. This suggests that distributive laws
  between double categories should correspond to suitable
  ``factorisation triple categories'', in the spirit of
  \v{S}t\v{e}p\'an; but we do not pursue this here.
\end{remark}

\section{Distributive laws from co-crossed double categories}\label{cocrossed}

In the rest of this paper, we exhibit a range of double categories,
including $\mathsf{T}$-spans, partial maps, topological
correspondences and polynomials, as the composites of distributive
laws. All of these arise from a single mechanism, whose input is a
double category $\C$ satisfying the following condition:
\begin{definition}\label{def:cocrossed}
  A double category $\C$ is \emph{co-crossed} when its codomain
  functor is a fibration, horizontal identities are cartesian, and
  cartesian cells are closed under horizontal composition.
\end{definition}

We now unpack our definition. A cell $\alpha$ of $\C$ as to the left
below is cartesian with respect to $\cod$ when, for each cell $\beta$
as centre below and each factorisation $n = kh$ of its right-hand
edge, there is a unique factorisation $\hat \beta$
such as shown to the right.
\[
\begin{tikzcd}
  \cdot \arrow[swap]{d}{g} \arrow{r}{v} \cell{dr}{\alpha} & \cdot \arrow{d}{k} \\
  \cdot \arrow[swap]{r}{w} & \cdot 
\end{tikzcd}
\quad \quad 
\begin{tikzcd}
  \cdot \arrow[swap]{dd}{m} \arrow{r}{u} \cell{ddr}{\beta} &
  \cdot \arrow{dd}{n} \cell{ddr}{=} &
  \cdot \arrow[swap]{d}{\hat m} \arrow{r}{u} \cell{dr}{\hat \beta} &
  \cdot \arrow{d}{h} \\
  & & 
  \cdot \arrow[swap]{d}{g} \arrow[r, "v" description] \cell{dr}{\alpha} &
  \cdot \arrow{d}{k} \\
  \cdot \arrow[swap]{r}{w} &
  \cdot &
  \cdot \arrow[swap]{r}{w} &
  \cdot 
\end{tikzcd}
\]
In this case, we call $\alpha$ a \emph{cartesian lift of $w$ at $k$},
and the fibration structure says that there is a cartesian lift of $w$
at $k$ for all $w,k$ in this configuration. Such cartesian cells are
always closed under binary and nullary \emph{vertical} composition; we
insist in our definition that they are moreover closed under nullary
and binary \emph{horizontal} composition.

\begin{example}
  For any category $\X$, the $\mathsf{cod}$-cartesian squares in the
  double category $\Sq(\X)$ are the pullback squares, which by the
  standard pullback properties are closed under horizontal and vertical
  composition. In particular, if $\X$ has pullbacks, then $\Sq(\X)$ is
  co-crossed.
\end{example}

In what follows, we assume our co-crossed double categories are provided with
a \emph{choice} of cartesian liftings:
\begin{equation*}
  \begin{tikzcd}
    & \cdot \arrow{d}{k} \cell{drr}{\mapsto} & & 
    \cdot \arrow{r}{w_k} \arrow[swap]{d}{k_w} \cell{dr}{\dagger_{w,k}} & \cdot \arrow{d}{k} \\
    \cdot \arrow[swap]{r}{w} & \cdot & & 
    \cdot \arrow[swap]{r}{w} & \cdot
  \end{tikzcd}
\end{equation*}
Such a choice is not unique, but is so to within isomorphism. More
precisely, if $\alpha$ as above is another cartesian lift of $w$ at
$k$, then there is a unique vertically invertible cell $\gamma$
witnessing the pasting equality:
\begin{equation*}
  \begin{tikzcd}
    \cdot \arrow{r}{w_k} \arrow[swap]{d}{k_w} \cell{dr}{\dagger_{w,k}} & \cdot \arrow{d}{k} \\
    \cdot \arrow[swap]{r}{w} & \cdot
  \end{tikzcd} \qquad = \qquad 
  \begin{tikzcd}
    \cdot \arrow{d} \arrow{r}{w_k} \cell{dr}{\gamma} & \cdot \idarr{d} \\
    \cdot \arrow[swap]{d}{g} \arrow[r, "v" description] \cell{dr}{\alpha} & \cdot \arrow{d}{k} \\
    \cdot \arrow[swap]{r}{w} & \cdot 
  \end{tikzcd}
\end{equation*}
We do \emph{not} assume that composites of chosen cartesian cells are
chosen. However, they will, by our earlier assumptions, be cartesian,
and so agree with the chosen lifts to within isomorphism, in the above
sense.

\cite[Construction~5.3.3]{weber2015} shows that each crossed internal
category $\C$ in a finitely complete $2$-category $\K$ gives rise to a
distributive law between $\C$ and the ``internal category of
squares'' (built using comma objects) of $\C_0$. When
$\K = \Cat^{\mathrm{co}}$, a crossed internal category is a co-crossed
double category which is strict, whose codomain fibration is split,
and whose chosen lifts are strictly preserved by horizontal
composition; the following is the analogue of Weber's construction
without these three restrictions.

We write $\C^{\mathrm{hop}}$ for the \emph{horizontal opposite} of a
double category $\C$, obtained by interchanging its domain and
codomain functors and thus the direction of its horizontal arrows and
cells.

\begin{proposition}\label{prop:fibred}
  Let $\C$ be a co-crossed double category equipped with a choice of
  cartesian liftings. There is a distributive law of $\Sq(\C_0)$ over
  $\C^{\mathrm{hop}}$ whose functor~\eqref{eq:deltafunctor} sends each
  formally composable pair $(f,u)$ to the horizontal and vertical
  domains $(u_f, f_u)$ of the chosen cartesian lift of $u$ at $f$.
\end{proposition}
\begin{proof}
  Following the notation of \cref{sec:distributive_laws} we are
  setting $\A=\Sq(\C_0)$ and $\B = \C^{\mathrm{hop}}$. Both double
  categories have the same objects and vertical arrows as $\C$. The
  horizontal arrows of $\A$ are also the vertical arrows of $\C$, and
  the squares are commuting squares of vertical arrows, as to the left
  in:
  \begin{equation*}
    \begin{tikzcd}
      \hcorner{a}{f}{g}{=} & b \arrow{d}{k} \\
      c \harrows{r}{h} & d
    \end{tikzcd} \qquad \qquad
    \begin{tikzcd}
      a \karrowl{r}{u} \arrow[swap]{d}{g} \kcell{dr}{\alpha} & b \arrow{d}{k} \\
      c \karrowsl{r}{v} & d
    \end{tikzcd}
  \end{equation*}
  In $\B$, by contrast, the horizontal arrows from $a$ to $b$ are the
  horizontal arrows of $\C$ from $b$ to $a$, while cells of $\B$ with
  left edge $g$ and right edge $k$ are cells of $\C$ with domain $k$
  and codomain $g$, as displayed above right. Here, and for the remainder
  of the proof, we notate horizontal arrows of $\C$ with a tick in
  order to minimise confusion between the two kinds of square.

  We begin by describing the functor $\delta$
  of~\eqref{eq:deltafunctor}. On objects, this is an assignment
  \[
    (
    \begin{tikzcd}
      a \harrow{r}{f} & b \karrowl{r}{u} & c
    \end{tikzcd}
    ) \qquad \mapsto \qquad
    (
    \begin{tikzcd}
      a \karrowl{r}{u_f} & \cdot \harrow{r}{f_u} & c
    \end{tikzcd}
    )
  \]
  where $f$ and $f_u$ are vertical arrows of $\C$, and $u$ and $u_f$
  are horizontal arrows. We obtain this by taking the horizontal
  and vertical domains of the chosen cartesian
  lift of $u$ at $f$:
  \[
    \begin{tikzcd}
      \cdot \arrowt{r}{u_f} \arrow[swap]{d}{f_u} \cart{dr} & a \arrow{d}{f} \\
      c \arrowts{r}{u} & b\rlap{ .}
    \end{tikzcd}
  \]
  On morphisms, $\delta$ is an assignment:
  \begin{equation*}
  \begin{tikzcd}
    \hcorner{a}{f}{g}{=} & b \karrowl{r}{u} \kcell{dr}{\beta} \arrow{d}{h} & c \arrow{d}{k} \\
    a' \harrows{r}{f'} & b' \karrowsl{r}{u'} & c'
  \end{tikzcd}
  \qquad
  \mapsto
  \qquad
  \begin{tikzcd}
    a \karrowl{r}{u_f} \kcell{dr}{\hat\beta} \arrow[swap]{d}{g} & \hcorner{\cdot}{f_u}{\ell}{=} & c \arrow{d}{k} \\
    a' \karrowsl{r}{u'_{f'}} &
    \cdot \harrows{r}{f'_{u'}} &
    c'\rlap{ .}
  \end{tikzcd}
  \end{equation*}
  Here $\hat\beta$ is obtained as follows. Pasting $\beta$ onto the
  chosen cartesian cell $\dagger_{u,f}$ gives the composite to the left
  below, whose right-hand edge is $hf$. Since the square marked $=$
  commutes, $hf = f'g$, so that this edge factors through $f'$;
  whereupon the universal property of the chosen cartesian lift
  $\dagger_{u',f'}$ of $u'$ at $f'$ yields a unique cell $\hat \beta$
  with right-hand edge $g$ satisfying:
  \begin{equation*}
    \begin{tikzcd}
      \cdot \arrowt{r}{u_f} \arrow[swap]{d}{f_u} \cell{dr}{\dagger_{u,f}} & a \arrow{d}{f} \\
      c \arrowt{r}{u} \arrow[swap]{d}{k} \cell{dr}{\beta} & b \arrow{d}{h} \\
      c' \arrowts{r}{u'} & b'
    \end{tikzcd}
    \qquad = \qquad
    \begin{tikzcd}
      \cdot \arrowt{r}{u_f} \arrow[swap]{d}{\ell} \cell{dr}{\hat\beta} & a \arrow{d}{g} \\
      \cdot \arrowt{r}{\smash{u'_{f'}}} \arrow[swap]{d}{f'_{u'}} \cell{dr}{\dagger_{u',f'}} & a' \arrow{d}{f'} \\
      c' \arrowts{r}{u'} & b'
    \end{tikzcd}
  \end{equation*}
  Comparing left-hand edges in this equality gives
  $f'_{u'} \ell = k f_u$, which is exactly the commuting square drawn
  to the right of $\hat\beta$ above.

  This describes the action of $\delta$ on objects and morphisms,
  and functoriality follows easily from the uniqueness part of the
  universal property of cartesian cells. We now turn to the coherence
  data $\cone$--$\cfour$ for the distributive law. In the following
  diagrams, which live in $\C$, we keep the names of composites from
  \cref{sec:distributive_laws}: thus $fg$ is the composite in $\A$,
  i.e., $gf$ in $\C_0$, and $uv$ is the composite in $\B$, i.e.,
  $v \ast u$ in $\C$.
    \begin{itemize}
    \item For each pair $(\id, u)$ we use the fact that (vertical)
      identities are cartesian for any fibration, and draw the diagram
      below, where $\cone$ is the unique isomorphism between cartesian
      cells.
    \[
      \begin{tikzcd}
        c \arrowt{r}{u} \arrow[swap]{d}{\ell} \cell{dr}{\cone} & a \idarr{d} \\
        \cdot \arrowt{r}{u_{\id}} \arrow[swap]{d}{\id_u} \cell{dr}{\dagger_{u,\id}} & a \idarr{d} \\
        c \arrowts{r}{u} & a
      \end{tikzcd}
      \qquad = \qquad
      \begin{tikzcd}
        c \arrowt{r}{u} \idarr{dd} \cell{ddr}{\id} & a \idarr{dd} \\
        & \\
        c \arrowts{r}{u} & a
      \end{tikzcd}
    \]
    This yields a pair of invertible cells of the required form.
    \[
      \begin{tikzcd}
        a \karrowl{r}{u} \idarr{d} \kcell{dr}{\cone} &
        c \harrow{r}{\id} \arrow{d}{\ell} \hcell{dr}{\cone} &
        c \idarr{d} \\
        a \karrowsl{r}{u_{\id}} & \cdot \harrows{r}{\id_u} & c
      \end{tikzcd}
    \]
    \item For each pair $(f, \id)$ we use the assumption that
      horizontal identities are cartesian, and draw the diagram below,
      where $\ctwo$ is the unique isomorphism between cartesian cells.
    \[
      \begin{tikzcd}
        \cdot \arrowt{r}{(\id_b)_f} \arrow[swap]{d}{\ell} \cell{dr}{\ctwo} & a \idarr{d} \\
        a \arrowt{r}{\id_a} \arrow[swap]{d}{f} \cell{dr}{\id_f} & a \arrow{d}{f} \\
        b \arrowts{r}{\id_b} & b
      \end{tikzcd}
      \qquad = \qquad
      \begin{tikzcd}
        \cdot \arrowt{r}{(\id_b)_f} \arrow[swap]{dd}{f_{\id}} \cell{ddr}{\dagger_{\id_b, f}} & a \arrow{dd}{f} \\
        & \\
        b \arrowts{r}{\id_b} & b
      \end{tikzcd}
    \]
    This yields a pair of invertible cells of the required form.
    \[
      \begin{tikzcd}
        a \karrowl{r}{(\id_b)_f} \idarr{d} \kcell{dr}{\ctwo} &
        \cdot \harrow{r}{f_{\id}} \arrow{d}{\ell} \hcell{dr}{\ctwo} &
        b \idarr{d} \\
        a \karrowsl{r}{\id_a} & a \harrows{r}{f} & b
      \end{tikzcd}
    \]
    \item For each triple $(f, g, u)$ we use the fact that cartesian
      cells for any fibration are closed under (vertical) composition,
      and draw the diagram below, where $\cthree$ is the unique
      isomorphism between cartesian cells.
    \[
      \begin{tikzcd}
        \cdot \arrowt{r}{(u_g)_f} \arrow[swap]{d}{\ell} \cell{dr}{\cthree} & a \idarr{d} \\
        \cdot \arrowt{r}{u_{fg}} \arrow[swap]{d}{(fg)_u} \cell{dr}{\dagger_{u,fg}} & a \arrow{d}{fg} \\
        d \arrowts{r}{u} & c
      \end{tikzcd}
      \qquad = \qquad
      \begin{tikzcd}
        \cdot \arrowt{r}{(u_g)_f} \arrow[swap]{d}{f_{u_g}} \cell{dr}{\dagger_{u_g,f}} & a \arrow{d}{f} \\
        \cdot \arrowt{r}{u_g} \arrow[swap]{d}{g_u} \cell{dr}{\dagger_{u,g}} & b \arrow{d}{g} \\
        d \arrowts{r}{u} & c
      \end{tikzcd}
    \]
    This yields a pair of invertible cells of the required form.
    \[
      \begin{tikzcd}
        a \karrowl{r}{(u_g)_f} \idarr{d} \kcell{dr}{\cthree} &
        \cdot \harrow{r}{f_{u_g}} \arrow{d}{\ell} \hcell{drr}{\cthree} &
        \cdot \harrow{r}{g_u} &
        d \idarr{d} \\
        a \karrowsl{r}{u_{fg}} &
        \cdot \harrows{rr}{(fg)_u} & &
        d
      \end{tikzcd}
    \]
    \item For each triple $(f, u, v)$ we use the assumption that
      cartesian cells are closed under horizontal composition, and
      draw the diagram below, where $\cfour$ is the unique isomorphism
      between cartesian cells.
    \[
      \begin{tikzcd}
        \cdot \arrowt{rr}{(uv)_f} \arrow[swap]{d}{\ell} \cell{drr}{\cfour} & & a \idarr{d} \\
        \cdot \arrowt{r}{v_{f_u}} \arrow[swap]{d}{(f_u)_v} \cell{dr}{\dagger_{v,f_u}} &
        \cdot \arrowt{r}{u_f} \arrow[swap]{d}{f_u} \cell{dr}{\dagger_{u,f}} &
        a \arrow{d}{f} \\
        d \arrowts{r}{v} & c \arrowts{r}{u} & b
      \end{tikzcd}
      \qquad = \qquad
      \begin{tikzcd}
        \cdot \arrowt{r}{(uv)_f} \arrow[swap]{dd}{f_{uv}} \cell{ddr}{\dagger_{uv,f}} & a \arrow{dd}{f} \\
        & \\
        d \arrowts{r}{uv} & b
      \end{tikzcd}
    \]
    This yields a pair of invertible cells of the required form.
    \[
      \begin{tikzcd}
        a \karrowl{rr}{(uv)_f} \idarr{d} \kcell{drr}{\cfour} & &
        \cdot \harrow{r}{f_{uv}} \arrow{d}{\ell} \hcell{dr}{\cfour} &
        d \idarr{d} \\
        a \karrowsl{r}{u_f} & \cdot \karrowsl{r}{v_{f_u}} &
        \cdot \harrows{r}{(f_u)_v} & d
      \end{tikzcd}
    \]
\end{itemize}
To see that $\cone$--$\cfour$ are natural and satisfy axioms
(\ref{eq:wone})--(\ref{eq:wfive}), one draws out the diagrams and
observes that each side of each equation is a composite of
isomorphisms determined by the universal properties of various
cartesian cells; each equation then holds by uniqueness of
isomorphisms between cartesian cells.
\end{proof}

In~\cite{weber2015}, the composite internal category obtained from a
distributive law constructed in the above manner plays an important
role. Weber concentrates on the case where $\K = \Cat(\E)$, so that a
crossed internal category in $\K$ is a crossed internal double
category in $\E$, and restricts the composite internal double category
back to an internal $2$-category, which he terms the
\emph{$2$-category of corners}. We follow this nomenclature, except
that we retain the double-categorical structure that Weber discards:

\begin{definition}\label{def:cnr}
  Let $\C$ be a co-crossed double category. Its \emph{double
    category of corners} $\Cnr \C$ is the composite double
  category $\C^{\mathrm{hop}} \bt \Sq(\C_0)$ of the distributive law of
  Proposition~\ref{prop:fibred}.
\end{definition}

\section{First example: \texorpdfstring{$\mathsf T$}{T}-spans}\label{sec:tspans}

\emph{$\mathsf T$-categories} were introduced by Burroni
in~\cite{Burroni1971T-categories}; they are a generalisation of
internal categories in a category $\E$ with pullbacks, wherein the
underlying source-target span $X_0 \leftarrow X_1 \rightarrow X_0$ is
replaced by a span of the form $TX_0 \leftarrow X_1 \rightarrow X_0$,
for $\mathsf{T}$ a cartesian monad on $\E$ (\cref{def:tspans}).
For suitable $\mathsf{T}$ and $\E$, the
$\mathsf{T}$-categories capture the (non-symmetric) multicategories;
the virtual double categories (= fc-multicategories)
of~\cite{Leinster2004Higher, Cruttwell2010A-unified}; and the globular
operads of
Batanin--Leinster~\cite{Batanin1998Monoidal,Leinster2004Higher}.

Just as internal categories in $\E$ are monads in the bicategory of
$\E$-spans, so $\mathsf{T}$-categories are monads in a bicategory of
\emph{$\mathsf{T}$-spans}. In this section, we show how the double
category of $\mathsf{T}$-spans arises from the machinery of
Section~\ref{cocrossed}. We start with the definitions, following the
conventions of~\cite{Leinster2004Higher}.

\begin{definition}\label{def:tspans}
  Let $\E$ be a category with pullbacks. A monad
  $\mathsf{T} = (T, \eta, \mu)$ on $\E$ is called \emph{cartesian}
  when $T$ preserves pullbacks, and each naturality square of $\eta$
  and $\mu$ is a pullback.
  In this situation, we define the double category of
  $\mathsf{T}$-spans as follows:
  \begin{itemize}[itemsep=0em]
  \item Its category of objects and vertical arrows is $\E$;
  \item Horizontal arrows from $a$ to $b$ are spans $f \colon Ta
    \leftarrow x \rightarrow b \colon g$;
  \item Cells bounded by the horizontal arrows $(f,g)$ and $(f',g')$
    and the vertical arrows $s$ and $t$ are maps $q$ fitting into a
    commuting diagram:
    \begin{equation}
      \label{eq:tspan-cell}
      \begin{tikzcd}
        Ta  \arrow[swap]{d}{Ts}
        & x \arrow{d}{q} \arrow[swap]{l}{f} \arrow{r}{g} 
        & b \arrow{d}{t} \\
        Ta'
        & x' \arrow{l}{f'} \arrow[swap]{r}{g'}
        & b'\rlap{ ;}
      \end{tikzcd}
    \end{equation}
  \item The horizontal identity on $a$ is the span $\eta_a \colon Ta
    \leftarrow a \rightarrow a \colon 1_a$;
  \item The horizontal composite of $\mathsf{T}$-spans $f \colon Ta
    \leftarrow x \rightarrow b \colon g$ and $h \colon Tb
    \leftarrow y \rightarrow c \colon k$ is given by the outer
    composite in:
    \begin{equation*}
      \begin{tikzcd}[sep = {3em, between origins}]
        &&& Tx \times_{Tb} y \ar[dl] \ar[dr] \pb[0.15]{dd} && \\
        && Tx \arrow[swap]{dl}{Tf} \arrow{dr}{Tg} & &
        y \arrow[swap]{dl}{h} \arrow{dr}{k} & \\
        &T^2a \arrow[swap]{dl}{\mu_a} && Tb && c\rlap{ ;} \\
        Ta
      \end{tikzcd}
    \end{equation*}
  \item The horizontal composite of cells is induced using the
    universal property of pullback;
  \item The associativity and unit constraint cells are induced using
    the universal property of pullback together with the cartesianness
    of the monad $\mathsf{T}$.
  \end{itemize}
\end{definition}
The last two points here are hazy---the reader can refer
to~\cite[Example 5.2.4]{Leinster2004Higher} for the details---but deliberately so,
as our goal is to obtain the double category of $\mathsf{T}$-spans in
a more principled way from the \emph{horizontal Kleisli double
  category} of $\mathsf{T}$.

To define this, note first that in a category with pullbacks, any
internal functor between internal categories has an internal
(bijective on objects, fully faithful) factorisation. In particular,
any double functor between strict double categories has such a
factorisation; so, given a monad $\mathsf{T}$ on the category $\E$,
we may consider the free functor
$F_{\mathsf{T}} \colon \E \rightarrow \E_{\mathsf{T}}$ into the
Kleisli category, the induced double functor
$\Sq(F_{\mathsf{T}}) \colon \Sq(\E) \rightarrow \Sq(\E_{\mathsf{T}})$,
and its factorisation
\begin{equation}\label{eq:boo-ff-kleisli}
  \Sq(\E) \rightarrow \KK \rightarrow \Sq(\E_{\mathsf{T}})\rlap{ .}
\end{equation}
The interposing double category is the \emph{horizontal Kleisli double
  category of $\mathsf{T}$}; it has objects and vertical arrows as in
$\E$; horizontal arrows $a \tor b$ given by $\E_\mathsf{T}$-arrows
from $a$ to $b$, i.e., maps $a \rightarrow Tb$ in $\E$; and cells as
to the left
given by commuting squares in $\E$ as to the right:
\begin{equation*}
  \begin{tikzcd}
    a \arrowt{r}{f} \arrow[swap]{d}{s} & b \arrow{d}{t} & & 
    a \arrow{r}{f} \arrow[swap]{d}{s} & Tb \arrow{d}{Tt} \\
    a' \arrowts{r}{f'} & b' & & 
    a' \arrow[swap]{r}{f'} & Tb'\rlap{ .}
\end{tikzcd}
\end{equation*}

\begin{lemma}
  \label{lem:cod-cartesian-pullbacks}
  Let $\mathsf{T}$ be a monad on $\E$.
  A cell of $\KK$ as to the left
  below is
  $\mathsf{cod}$-cartesian just when it is a pullback in $\E$.
  \begin{equation*}
    \begin{tikzcd}
      a'  \arrow{r}{f'} \arrow[swap]{d}{u}  & 
      Tb' \arrow{d}{Tv} \\
      a'' \arrow[swap]{r}{f''} & 
      Tb''
    \end{tikzcd} \qquad 
    \begin{tikzcd}
      a  \arrow{r}{f} \arrow[dotted]{d}{s} \arrow[bend right, swap]{dd}{w}& 
      Tb \arrow[swap]{d}{Tt} \\
      a'  \arrow{r}{f'} \arrow{d}{u}  & 
      Tb' \arrow[swap]{d}{Tv} \\
      a'' \arrow[swap]{r}{f''} & 
      Tb''
    \end{tikzcd} \qquad 
    \begin{tikzcd}
      a  \arrow[bend left]{dr}{x}  \arrow[dotted]{d}{s} \arrow[bend right, swap]{dd}{w}  \\
      a'  \arrow{r}{f'} \arrow{d}{u}  & 
      Tb' \arrow[swap]{d}{Tv} \\
      a'' \arrow[swap]{r}{f''} & 
      Tb''\rlap{ .}
    \end{tikzcd}
  \end{equation*}
\end{lemma}
\begin{proof}
  To say that this square is $\mathsf{cod}$-cartesian is to say that,
  given the solid part of the diagram middle above, there is a unique
  $s$ as shown making everything commute. On the other hand, to say
  that it is a pullback is to make the corresponding claim about the
  diagram right above. Clearly the latter implies the former, and the
  former implies the latter on taking $t = 1_{b'}$.
\end{proof}

\begin{theorem} \label{thm:tspans} If $\mathsf{T}$ is a cartesian
  monad on a category $\E$ with pullbacks, then
  $\KK$ is co-crossed, and its double category of
  corners $\KK^{\mathrm{hop}} \bt \Sq(\E)$ is the
  double category of $\mathsf{T}$-spans.
\end{theorem}
\begin{proof}
  Since $\E$ has all pullbacks, it follows immediately from
  \cref{lem:cod-cartesian-pullbacks} that the codomain functor of
  $\KK$ is a fibration. To show stability of cartesian squares under horizontal
  identities and composition, consider the following two diagrams. To
  the left is the horizontal identity cell on $s$, which is a pullback
  because $\eta$ is cartesian. To the right is shown the horizontal
  composite of two cartesian squares, which is itself cartesian
  because $T$ preserves pullbacks and $\mu$ is cartesian.
\[
\begin{tikzcd}
    a  \arrow{r}{\eta_a} \arrow[swap]{d}{s} & 
    Ta \arrow{d}{Ts} \\
    b \arrow[swap]{r}{\eta_b} & 
    Tb
\end{tikzcd}
\qquad\qquad
\begin{tikzcd}
    a  \arrow{r}{f} \arrow[swap]{d}{s} \pb{dr} &
    Tb  \arrow{r}{Tg} \arrow{d}{Tt} \pb{dr} &
    T^2c \arrow{r}{\mu_c} \arrow{d}{T^2u} \pb{dr} &
    Tc \arrow{d}{Tu} \\
    a'  \arrow[swap]{r}{f'}  &
    Tb'  \arrow[swap]{r}{Tg'} &
    T^2c' \arrow[swap]{r}{\mu_{c'}} &
    Tc' \rlap{ .}
\end{tikzcd}
\]
This proves that $\KK$ is co-crossed, and so admits
a double category of corners $\KK^{\mathrm{hop}} \bt \Sq(\E)$. By
definition, this has as horizontal arrows from $a$ to $b$, spans of
the form $f \colon Ta \leftarrow x \rightarrow b \colon g$, and as
cells, diagrams of the form~\eqref{eq:tspan-cell}; moreover, we can
read off from~\eqref{eq:hcomp-bowtie} that the composition of such
horizontal arrows and cells is exactly that of
Definition~\ref{def:tspans}; and the coherence cells agree, since on
both sides they are the comparison isomorphisms between pullbacks.
\end{proof}

In fact, the first part of this theorem can be strengthened to an
if-and-only-if:

\begin{theorem}\label{thm:T-cart}
  Let $\mathsf{T}$ be any monad on any category $\E$. The double category
  $\KK$ is co-crossed if and only if $\mathsf{T}$ is cartesian and
  $\E$ has pullbacks.
\end{theorem}
\begin{proof}
  For the direction not covered by \cref{thm:tspans}, suppose that
  $\KK$ is co-crossed. Firstly its horizontal identity cells,
  which are cells of the following form in $\E$:
  \begin{equation*}
    \begin{tikzcd}
      x \arrow{r}{\eta_x} \arrow[swap]{d}{v} &
      Tx \arrow{d}{Tv} \\
      z \arrow[swap]{r}{\eta_z} &
      Tz
    \end{tikzcd}
  \end{equation*}
  are $\cod$-cartesian, and so, by \cref{lem:cod-cartesian-pullbacks},
  are pullbacks in $\E$. Thus $\eta \colon 1_\E \Rightarrow T$ is a
  cartesian natural transformation. Next, write
  $\epsilon_x \colon Tx \tor x$ for the horizontal arrow of $\KK$
  given by $1_{Tx}$, i.e., the counit of the Kleisli adjunction, and
  observe that the trivial pullback square in $\E$ to the right of:
  \begin{equation*}
    \begin{tikzcd}
      Tx \arrowt{r}{\epsilon_x} \arrow[swap]{d}{Tv} &
      x \arrow{d}{v} &&
      Tx \arrow{r}{\id} \arrow[swap]{d}{Tv} &
      Tx \arrow{d}{Tv} \\
      Tz \arrowts{r}{\epsilon_z} &
      z&&
      Tz \arrow[swap]{r}{\id} &
      Tz
    \end{tikzcd}
  \end{equation*}
  corresponds by \cref{lem:cod-cartesian-pullbacks} to a $\cod$-cartesian
  cell in $\KK$, as to the left. Since $\cod$-cartesian cells
  are closed under horizontal composition in $\KK$, the composite cell
  to the left in:
  \begin{equation*}
    \begin{tikzcd}
      T^2x \arrowt{r}{\epsilon_{Tx}} \arrow[swap]{d}{T^2v} &
      Tx \arrowt{r}{\epsilon_x} \arrow[swap]{d}{Tv} &
      x \arrow{d}{v} & & 
      T^2x \arrow{r}{\mu_x} \arrow[swap]{d}{T^2v} &
      Tx \arrow{d}{Tv} \\
      T^2z \arrowts{r}{\epsilon_{Tz}} &
      Tz \arrowts{r}{\epsilon_z} &
      z&&
      T^2z \arrow[swap]{r}{\mu_z} &
      Tz
    \end{tikzcd}
  \end{equation*}
  is thus cartesian; so the corresponding square in $\E$ is a
  pullback. But this square is a naturality square of $\mu$, as
  displayed right above, which is thus a cartesian natural
  transformation.

  Next we show that $\E$ has pullbacks. Given maps
  $t \colon y \rightarrow z \leftarrow x \colon v$ in $\E$, take the
  cartesian lift of $\eta_z t$ at $v$, as to the left
  below, and look at the corresponding pullback in $\E$: the outer
  square to the right.
  \begin{equation*}
    \begin{tikzcd}
      \cdot \arrowts{r}{s} \arrow[swap]{d}{u} 
      &
      x \arrow{d}{v} &&
      \cdot \arrow[bend left]{rr}{s} \arrow[swap]{d}{u} \arrow[dotted]{r}
      & x \arrow{r}{\eta_x} \arrow[swap]{d}{v} &
      Tx \arrow{d}{Tv} \\
      y \arrowts{r}{\eta_z t} &
      z&&
      y \arrow[swap]{r}{t} & 
      z \arrow[swap]{r}{\eta_z} & 
      Tz
    \end{tikzcd}
  \end{equation*}
  Since $\eta$ is cartesian, $s$ factorises through $\eta_x$ as shown,
  and then the left square is a pullback because pullbacks are
  left-cancellable.
  Finally, we show that $T$ preserves pullbacks. Suppose that
  $g {\hat v} = v {\hat g}$ is a pullback in $\E$ and observe that the
  following two diagrams are pullbacks.
  \begin{equation*}
    \begin{tikzcd}
      Tw \arrow{r}{\id} \arrow[swap]{d}{T{\hat v}} &
      Tw \arrow{d}{T{\hat v}} \\
      Ty \arrow[swap]{r}{\id} &
      Ty
    \end{tikzcd}
    \qquad\qquad
    \begin{tikzcd}
      w \arrow{r}{\hat g} \arrow[swap]{d}{\hat v} &
      x \arrow{r}{\eta_x} \arrow[swap]{d}{v} &
      Tx \arrow{d}{Tv} \\
      y \arrow[swap]{r}{g} & 
      z \arrow[swap]{r}{\eta_z} & 
      Tz\rlap{ .}
    \end{tikzcd}
  \end{equation*}
  So the corresponding cells of $\KK$ are cartesian. Composing these
  horizontally, as to the left below, yields another cartesian square,
  whose corresponding pullback square in $\E$ is the $T$-image of our
  original pullback, as right below.
  \begin{equation*}
    \begin{tikzcd}
      Tw \arrowt{r}{\epsilon_w} \arrow[swap]{d}{T\hat v} &
      w \arrowt{r}{\eta_x {\hat g}} \arrow[swap]{d}{\hat v} &
      x \arrow{d}{v} &&
      Tw \arrow{r}{T\hat g} \arrow[swap]{d}{T\hat v} &
      Tx \arrow[swap]{d}{Tv} \\
      Ty \arrowts{r}{\epsilon_y} &
      y \arrowts{r}{\eta_z g} &
      z &&
      Ty \arrow[swap]{r}{Tg} & 
      Tz 
    \end{tikzcd} \qedhere
  \end{equation*}
\end{proof}

In particular, taking $\mathsf{T}$ to be the identity monad in the
preceding two results, we have:
\begin{corollary} \label{cor:spans} A category $\E$ has pullbacks if
  and only if $\Sq(\E)$ is co-crossed; in this case, its double
  category of corners $\Sq(\E)^\mathrm{hop} \bt \Sq(\E)$ is the 
  double category $\Span(\E)$ of spans in $\E$.
\end{corollary}

\section{Calibrated categories}\label{sec:calibrated_cats}

Our second main example, given in \cref{sec:polynomial_spans},
reconstructs the double category $\Poly(\E)$ of polynomials in a
locally cartesian closed category $\E$ as the double category of
corners of a co-crossed double category; as explained in the introduction, it does
so by way of Ross's calibrated bicategories~\cite{Street2020}. As a
warm-up, in this section we re-express calibrations of categories, in
the sense of B\'enabou~\cite{benabou1975theories}, in terms of
co-crossed double categories.

The starting point is a generalisation of the construction
of~\eqref{eq:boo-ff-kleisli}. To any identity-on-objects functor
$J \colon \A \rightarrow \M$ between categories, we may associate a
double category $\mathsf{H}J$ as the interposing double category in the
internal (bijective on objects, fully faithful) factorisation of
$\Sq(J)$:
\begin{equation*}
  \Sq(\A) \rightarrow \mathsf{H}J \rightarrow \Sq(\M)\rlap{ .}
\end{equation*}
Explicitly, $\mathsf{H}J$ has the same objects and vertical arrows
as $\A$, as horizontal arrows, those of $\M$, and as cells,
commuting squares in $\M$ whose vertical edges are of the form $f_*$.
In particular, when
$J = F_\mathsf{T} \colon \E \rightarrow \E_\mathsf{T}$ is the free
functor into a Kleisli category, we re-find the horizontal Kleisli
double category $\KK$ of the preceding section. We will characterise
$\mathsf{H}J$'s being co-crossed in terms of a generalisation of
B\'enabou's notion of calibration:

\begin{definition}\label{def:calibration}
  Consider an identity-on-objects functor $J \colon \A \to \M$, and
  write $f_\ast \defeq J f$ for each morphism of $\A$. We say that
  $J$ is a \emph{calibration} of $\M$ when it satisfies the following
  conditions:
  \begin{enumerate}[label=(R\arabic*), start=1]
  \item \label{Rlc} Every map of $\A$ is cartesian with respect to
    $J$. More precisely, for all $p \in \M(x,y)$, $q \in \A(y,z)$ and
    $r \in \A(x,z)$, if $r_* = q_*p$ then there is a unique $p' \in \A(x,y)$
    with $p = p'_*$ and $r = qp'$.
  \item \label{Rpb} For all $n \in \M(y,z)$ and $p \in \A(x,z)$
    there exist $\tilde n \in \M(w,x)$ and $\tilde p \in \A(w,y)$
    making the following diagram a pullback in $\M$:
    \begin{equation*}
      \begin{tikzcd}
        w \arrow{r}{\tilde n} \arrow[swap]{d}{\tilde p_*} &
        x \arrow{d}{p_*} \\
        y \arrow[swap]{r}{n} &
        z\rlap{ .}
      \end{tikzcd}
    \end{equation*}
  \end{enumerate}
\end{definition}
To recapture B\'enabou's original definition
in~\cite{benabou1975theories}, one takes the functor $J$ to be the
inclusion of a wide subcategory, so that the data of $J$ is that of a
class of maps $\P$ in $\M$ closed under binary and nullary
composition. Condition \ref{Rpb} then asserts stability of this class under
pullback, while \ref{Rlc} says that it satisfies the left two-for-three
property: if $r = qp$ in $\M$ with $r,q \in \P$, then $p \in \P$ also.
(B\'enabou also assumes the existence of a terminal object in $\M$,
and that the class $\P$ contains the isomorphisms, but these are
inessential details.)

\begin{theorem}\label{lemma:cocrossd-calib1}
An identity-on-objects functor $J \colon \A \to \M$ is a calibration if and only if $\mathsf{H}J$ is co-crossed.
\end{theorem}
\begin{proof}
  ($\Rightarrow$) Suppose that $J$ is a calibration. To show that
  $\cod$ is a fibration, let $p \in \A(x,z)$ and $n \in \M(y,z)$. By
  \ref{Rpb}, there are $\tilde n \in \M(w,x)$ and
  $\tilde p \in \A(w,y)$ making the lower square below a pullback in
  $\M$:
  \begin{equation*}
    \begin{tikzcd}
      u \arrow{r}{m} \arrow[dotted]{d}{s} \arrow[bend right, swap]{dd}{q_*} &
      v \arrow[swap]{d}{r_*} \arrow[bend left]{dd}{(pr)_*} \\
      w  \pb{dr} \arrow{r}{\tilde n} \arrow{d}{\tilde p_*}  &
      x \arrow[swap]{d}{p_*} \\
      y \arrow[swap]{r}{n} &
      z
    \end{tikzcd}
  \end{equation*}
  We claim that this square is a $\cod$-cartesian cell of
  $\mathsf HJ$. So let $m \in \M(u,v)$, $q \in \A(u,y)$ and
  $r \in \A(v,x)$ make the outer square commute. As the lower square
  is a pullback, there is a unique $s \in \M(u,w)$ with
  $\tilde p_* s = q_*$ and $\tilde n s = r_* m$; and as
  $q_* = \tilde p_* s$, \ref{Rlc} gives a unique $s' \in \A(u,w)$
  with $s'_* = s$ and $\tilde p s' = q$. So the outer square factors
  uniquely through the lower one, as required.

  This argument shows that every square of $\mathsf HJ$ which is a
  pullback in $\M$ is $\cod$-cartesian; conversely, any cartesian
  square differs from such a pullback by an invertible cell, and so is
  itself a pullback. Since pullback squares include the horizontal
  identity squares and are closed under horizontal pasting,
  $\mathsf{H}J$ is co-crossed.

  ($\Leftarrow$) Suppose that $\mathsf{H}J$ is co-crossed. As
  $\mathsf HJ \to \Sq(\M)$ is internally fully faithful, a cell of
  $\mathsf HJ$ is just a square of $\M$ together with a lifting of
  its vertical edges to $\A$, and so is determined by its boundary. We
  write $\dagger$ for a $\cod$-cartesian cell.

  For \ref{Rlc}, let $p \in \M(x,y)$, $q \in \A(y,z)$ and
  $r \in \A(x,z)$ with $r_* = q_*p$. This gives a square of
  $\mathsf{H}J$ as around the outside to the left below. The
  horizontal identity square on $q$ is cartesian, so the outer square
  factors through it via a unique cell as at the top. The left edge of
  this cell is a map $p' \in \A(x,y)$ with $p'_* = p$ and $qp' = r$,
  and is the unique such $p'$, since any other would give a second
  factorisation.
  \begin{equation*}
    \begin{tikzcd}
      x \arrow{r}{p} \arrow[dotted]{d}{p'_*} \arrow[bend right, swap]{dd}{r_*} &
      y \arrow[swap]{d}{\id_*} \arrow[bend left]{dd}{q_*} \\
      y  \cart{dr} \arrow{r}{\id} \arrow{d}{q_*}  &
      y \arrow[swap]{d}{q_*} \\
      z \arrow[swap]{r}{\id} &
      z
    \end{tikzcd}
    \qquad\qquad
    \begin{tikzcd}
      v \arrow{rr}{o} \arrow[swap, bend right]{dd}{\id_*} \arrow[dotted]{d}{t_*} &&
      x \arrow[bend left]{dd}{p_*} \arrow[swap]{d}{\id_*} \\
      u \cart{dr} \arrow{r}{\tilde m} \arrow{d}{{\tilde {\tilde p}}_*} &
      w \cart{dr} \arrow{r}{\tilde n} \arrow{d}{{\tilde p}_*} &
      x \arrow[swap]{d}{p_*} \\
      v \arrow[swap]{r}{m} &
      y \arrow[swap]{r}{n} &
      z
    \end{tikzcd}
  \end{equation*}
  For \ref{Rpb}, let $n \in \M(y,z)$ and $p \in \A(x,z)$, and let
  $\tilde n \in \M(w,x)$ and $\tilde p \in \A(w,y)$ form the cartesian
  lift of $n$ at $p$, as above right. We claim this
  square is a pullback in $\M$. So let $m \in \M(v,y)$ and
  $o \in \M(v,x)$ with $p_*o = nm$, and let $\tilde m \in \M(u,w)$ and
  $\tilde{\tilde p} \in \A(u,v)$ form the cartesian lift of $m$
  at $\tilde p$, as to the left. By hypothesis, the horizontal composite of
  these two cartesian squares is again cartesian, so that the outer
  square must factor through it. Thus, there is a unique
  $t \in \A(v,u)$ with $\tilde{\tilde p} t = \id_v$ and
  $\tilde n \tilde m t_* = o$. Now $h \defeq \tilde m t_* \in \M(v,w)$
  satisfies $\tilde n h = o$ and
  $\tilde p_* h = m \tilde{\tilde p}_* t_* = m$.

  Now suppose $h' \in \M(v,w)$ also satisfied $\tilde n h' = o$ and
  $\tilde p_* h' = m$. The square with top $h'$, bottom $m$ and
  vertical edges $\id_v$ and $\tilde p$ is then a cell with right edge
  $\tilde p = \tilde p \circ \id_w$; factoring it through the left
  cartesian square gives $t' \in \A(v,u)$ with
  $\tilde{\tilde p} t' = \id_v$ and $\tilde m t'_* = h'$. Now
  $\tilde n \tilde m t'_* = \tilde n h' = o$, so $t'$ satisfies the
  same conditions as $t$, whence $t' = t$ and $h' = h$. Thus the
  cartesian lift is a pullback in $\M$, and \ref{Rpb} holds.
\end{proof}

For us, the import of \cref{lemma:cocrossd-calib1} is that every
calibration $J$ determines a double category of corners
$\Cnr ({\mathsf{H}J})$ via the induced distributive law, whose
horizontal arrows $x \tor y$ are spans $x \leftarrow w \rightarrow y$
with left leg in $\M$ and right leg in $\A$, and whose composition is
by pullback. In several natural examples the leg in $\A$ is
conventionally placed on the left; we remedy this by passing to the
horizontal opposite $\Cnr ({\mathsf{H}J})^{\mathrm{hop}}$.

\begin{example}
  Let $\Set_{\subseteq}$ be the category of sets and subset
  inclusions, and let $J \colon \Set_{\subseteq} \to \Set$ be the
  obvious inclusion functor; it is a calibration. The underlying
  bicategory of $\Cnr ({\mathsf{H}J})^{\mathrm{hop}}$ is the locally posetal
  bicategory of partial functions.
\end{example}

\begin{example}
  Let $\Top_{lh}$ be the category of topological spaces and local
  homeomorphisms and let $J \colon \Top_{lh} \to \Top$ be the obvious
  inclusion functor; it is a calibration. The underlying bicategory of
  $\Cnr ({\mathsf{H}J})^{\mathrm{hop}}$ has as horizontal arrows $X \tor Y$ the
  topological spans $X \leftarrow E \to Y$ whose left leg is a local
  homeomorphism; these are the \emph{topological correspondences}
  of~\cite{Miller2024}, and for $X = Y$ they are Katsura's
  \emph{topological graphs}~\cite{Katsura2004}. Here \ref{Rlc} amounts
  to the familiar fact that a map of \'etale spaces over $X$ is itself
  \'etale, so a globular cell is simply a continuous map $E \to E'$
  commuting with both legs and is automatically a local homeomorphism.
\end{example}

\begin{example}
  Let $\DFib$ be the category of small categories and discrete
  fibrations and let $J \colon \DFib \to \Cat$ be the obvious
  inclusion functor; it is a calibration. Since
  $\DFib/\C \simeq [\C^\op, \Set]$, a horizontal arrow
  $\C \tor \D$ in $\Cnr ({\mathsf{H}J})^{\mathrm{hop}}$ is a presheaf $F$ on $\C$
  together with a functor $\el F \to \D$, and a globular cell is a map
  of presheaves commuting with the functors to $\D$. 
\end{example}

Finally let us observe that, by combining the results of the last two
sections, we obtain:

\begin{theorem}\label{thm:cart-equiv}
  Let $\mathsf{T}$ be a monad on a category $\E$. The following are
  equivalent:
\begin{enumerate}[label=(\roman*),itemsep=0em]
    \item $\E$ has all pullbacks and $\mathsf{T}$ is cartesian;
    \item $\KK$ is co-crossed;
    \item $F_\mathsf{T} \colon \E \rightarrow \E_\mathsf{T}$ is a calibration.
\end{enumerate}
\end{theorem}
\begin{proof}
(i) $\iff$ (ii) by \cref{thm:T-cart}, while
(ii) $\iff$ (iii) by \cref{lemma:cocrossd-calib1} and the
identification of $\mathsf{H}(F_\mathsf{T})$ with $\KK$.
\end{proof}

\section{Calibrated bicategories}
\label{sec:calibr-bicat}
Having understood the relation of categorical calibrations to
co-crossed double categories and distributive laws, we now do the same
for the bicategorical calibrations of~\cite{Street2020}. Our
definitions are not quite the same as Ross's, but the differences, as
we will see, are fairly minor.

\begin{definition}
  \label{def:groupoid-fibration}
  A functor between categories $p \colon \E \rightarrow \B$ is a
  \emph{groupoid fibration} when the following square is a bipullback
  of categories
  \begin{equation*}
    \begin{tikzcd}
        \E^\mathbf{2} \arrow{r}{\cod} \arrow[swap]{d}{p^\mathbf{2}} &
        \E \arrow{d}{p}\\
        \B^\mathbf{2} \arrow[swap]{r}{\cod} &
        \B \rlap{ .}
      \end{tikzcd}
  \end{equation*}
  A morphism $p \colon y \rightarrow z$ in a bicategory $\M$ is a
  groupoid fibration when for all $x \in \M$, the functor
  $\M(x, p) \colon \M(x,y) \rightarrow \M(x,z)$ is a groupoid
  fibration.
\end{definition}

This repackages the definition of~\cite[\S3]{Street2020}: essential
surjectivity of the comparison from $\E^\mathbf{2}$ into the
pseudo-pullback of $p^\mathbf{2}$ and $\cod$ is Ross's condition~(i),
and its full fidelity says that every morphism of $\E$ is cartesian for
$p$, his condition~(ii).

\begin{definition}\label{def:calibratedbicategory}
  Let $J \colon \A \to \M$ be an identity-on-objects pseudofunctor
  from a category to a bicategory. Writing $f_* \defeq Jf$ for
  any $1$-cell $f$ in $\A$, and $\gamma \colon (gf)_* \cong g_* f_*$
  for the coherence data of $J$, we say that $J$ is a
  \emph{calibration} of $\M$ when it satisfies the following:
  \begin{enumerate}[label=(C\arabic*), start=1,itemsep=0em]
  \item \label{C1} For each $q \in \A(y,z)$ and each object $x$, the
    following square is a bipullback in $\Cat$:
    \begin{equation*}
      \begin{tikzcd}
        \A(x,y) \arrow{r}{q \circ (\thg)} \arrow[swap]{d}{J} &
        \A(x,z) \arrow{d}{J}\\
        \M(x,y) \arrow[swap]{r}{q_* \circ (\thg)} &
        \M(x,z)
        \arrow[from=1-1,to=2-2,phantom,"\cong_\gamma"{description}]\rlap{ ;}
      \end{tikzcd}
    \end{equation*}
  \item \label{C2} Each $f_*$ is a groupoid fibration in $\M$;
  \item \label{C3} For all $n \in \M(y,z)$ and $p \in \A(x,z)$ there
    exist $\tilde n \in \M(w,x)$ and $\tilde p \in \A(w,y)$ and an
    invertible $2$-cell making the following diagram a bipullback in
    $\M$:
    \begin{equation*}
      \begin{tikzcd}
        w \arrow{r}{\tilde n} \arrow[swap]{d}{{\tilde p}_*} &
        x \arrow{d}{p_*}\\
        y \arrow[swap]{r}{n} &
        z \arrow[from=1-1,to=2-2,phantom,"\cong"{description}]\rlap{ .}
      \end{tikzcd}
    \end{equation*}
  \end{enumerate}
\end{definition}

\begin{remark}
  We have deliberately aligned our definition of calibrated bicategory
  with that of calibrated category. Indeed, if $\M$ in
  Definition~\ref{def:calibratedbicategory} is a $1$-category viewed as
  a locally discrete $2$-category then a pseudofunctor
  $J \colon \A \to \M$ is just a functor, and being a calibration in
  the sense of
  Definition~\ref{def:calibratedbicategory} reduces to being one in
  the sense of Definition~\ref{def:calibration}. In particular, the
  bipullback in \ref{C1} is then a pullback of sets, which is exactly
  \ref{Rlc}; bipullbacks in a $1$-category are ordinary pullbacks, so
  \ref{C3} is \ref{Rpb}; and \ref{C2} becomes trivial because all
  arrows of a $1$-category are groupoid fibrations.
\end{remark}

\begin{remark}\label{rem:C1}
  \ref{C1} asserts that each map of $\A$ is cartesian with respect to
  $J$, in the bicategorical sense; its universal property unpacks into the following two statements.
  \begin{itemize}[itemsep=0em]
  \item \emph{Existence}: for all $p \in \M(x,y)$, $q \in \A(y,z)$,
    $r \in \A(x,z)$ and invertible $\tau\colon r_* \Rightarrow q_*p$,
    there exist $p' \in \A(x,y)$ and an invertible
    $\tau' \colon p'_\ast \Rightarrow p$ with $r = qp'$ and
    $(q_\ast \tau')\gamma = \tau$.
  \item \emph{Uniqueness}: given $k, k' \in \A(x,y)$, $q \in \A(y,z)$
    and $\beta \colon k_* \To k'_*$, if
    $qk = qk'$ and 
    $(q_*\beta)\gamma = \gamma$ then $k = k'$ and $\beta$ is the
    identity.
  \end{itemize}
  In many of our examples, $J$ is a locally fully faithful
  pseudofunctor, so that each vertical map in~\ref{C1} is fully
  faithful; in this case, both the uniqueness clause and the $r = qp'$
  part of the existence are forced.
\end{remark}

\begin{remark}\label{rem:Jaxioms}
  We now compare our notion of bicategorical calibration
  to~\cite[Definition~8.1]{Street2020}. The starting point there is a
  class $\P$ of $1$-cells in $\M$, termed \emph{neat}. Given a
  calibration in our sense, we may define a $1$-cell $f$ of $\M$ to be
  neat when $f \cong p_\ast$ for some $p \in \A$. On doing so, we
  obtain a calibration in Ross's sense, except for the inessential
  point that our neat $1$-cells need not include all equivalences in
  $\M$. However, a calibration in Ross's sense need not arise from a
  $J \colon \A \rightarrow \M$ as in \cref{def:calibratedbicategory},
  for the following two reasons.

  Firstly, the neat maps need not arise as the essential image of an
  identity-on-objects pseudofunctor out of a $1$-category. For us,
  this choice is forced by our decision to work with double categories
  and not with something which is weakly associative in both
  directions, say \emph{doubly weak double
    categories}~\cite{FairbanksShulman2025} or \emph{double
    bicategories}~\cite{Verity1992}. That level of generality is not
  required for our examples so we have not explored it.

  Secondly, Ross restricts \ref{C3} to the case where not only is
  $p_\ast$ neat, but also $n$ lies in the class of so-called
  \emph{right lifters}. We drop this restriction because our ambit is
  narrower: we only describe the bicategory of polynomials $\Poly(\E)$
  for $\E$ locally cartesian closed. For $\E$ merely with pullbacks,
  as in~\cite{weber_polynomials}, one must restrict the morphisms of
  the bicategory of polynomials, and Ross does so precisely by way of
  the right lifters.

  The reader who wanted to recapture Ross's more general case in our
  framework could easily do so: simply take as extra data a class $\L$
  of $1$-cells of $\M$, containing the identities and closed under
  composition and invertible $2$-cells; require $n \in \L$ in
  \ref{C3}; ask that the resulting $\tilde n$ lie in $\L$; and take
  the horizontal arrows of $\mathsf{H}J$ to be the members of $\L$.
  This adds no technical complexity, but makes the story less clean;
  whence we leave such adjustments to the interested reader.
\end{remark}

The following definition generalises the construction of the double
category $\mathsf{H}J$ associated to an identity-on-objects functor
$J$ to the bicategorical case. In this situation we can no longer define
$\mathsf{H}J$ via a factorisation, but must instead do so by hand:

\begin{definition}\label{def:Jdouble}
  Let $J \colon \A \rightarrow \M$ be an identity-on-objects
  pseudofunctor from a category to a bicategory. We write
  $\mathsf{H}J$ for the double category wherein:
  \begin{itemize}[itemsep=0em]
  \item objects are those of $\A$;
  \item vertical arrows are arrows of $\A$;
  \item horizontal arrows are $1$-cells of $\M$;
  \item cells are $2$-cells in $\M$ of the form:
    \begin{equation*}
      \begin{tikzcd}
        x \arrow[""{name=top, anchor=center, inner sep=0}]{r}{m} \arrow[swap]{d}{f_*} &
        y \arrow{d}{g_*}\\
        x' \arrow[""{name=bot, anchor=center, inner sep=0}, swap]{r}{m'} &
        y'
        \arrow[Rightarrow, from=bot, to=top, shorten <=1.0em, shorten >=1.0em, "\kappa"']
        \rlap{\, .}
      \end{tikzcd}
    \end{equation*}
  \end{itemize}
  Vertical composition and identities are determined by the following
  pastings in $\M$, where the unlabelled cells are coherence data for $J$.
  \begin{equation*}
    \begin{tikzcd}[column sep = 3em]
      x \arrow[""{name=top, anchor=center, inner sep=0}]{r}{m} \arrow{d}{f_*} \arrow[swap, bend right]{dd}{(f'f)_*}&
      y \arrow[swap]{d}{g_*} \arrow[bend left]{dd}{(g'g)_*}\\
      x' \arrow[""{name=mid, anchor=center, inner sep=0}, "m'" description]{r}{} \arrow{d}{f'_*} &
      y' \arrow[swap]{d}{g'_*}\\
      x'' \arrow[""{name=bot, anchor=center, inner sep=0},swap]{r}{m''} &
      y''
      \arrow[Rightarrow, from=mid, to=top, shorten <=1.0em, shorten >=1.0em, "\kappa"']
      \arrow[Rightarrow, from=bot, to=mid, shorten <=1.0em, shorten >=1.0em, "\kappa'"']
    \end{tikzcd}
    \qquad \qquad
    \begin{tikzcd}[column sep = 5em]
      x \arrow[""{name=top, anchor=center, inner sep=0}]{r}{m} \arrow{d}{\id} \arrow[bend right, swap]{d}{\id_*} &
      y \arrow[swap]{d}{\id} \arrow[bend left]{d}{\id_*} \\
      x \arrow[""{name=bot, anchor=center, inner sep=0}, swap]{r}{m} &
      y
      \arrow[Rightarrow, from=bot, to=top, shorten <=1.0em, shorten >=1.0em, "\lambda^{-1}\rho"']
      \rlap{\, .}
    \end{tikzcd}
  \end{equation*}
  Horizontal composition and identities are determined by the following pastings in $\M$.
  \begin{equation*}
    \begin{tikzcd}
      x \arrow[""{name=topL, anchor=center, inner sep=0}]{r}{m} \arrow[swap]{d}{f_*} &
      y \arrow[""{name=topR, anchor=center, inner sep=0}]{r}{n} \arrow{d}{g_*} &
      z \arrow{d}{h_*}  & &
      x \arrow[""{name=top, anchor=center, inner sep=0}]{rr}{\id} \arrow[swap]{d}{f_*} & &
      x \arrow{d}{f_*} \\
      x' \arrow[""{name=botL, anchor=center, inner sep=0}, swap]{r}{m'} &
      y' \arrow[""{name=botR, anchor=center, inner sep=0}, swap]{r}{n'} &
      z' 
      \arrow[Rightarrow, from=botL, to=topL, shorten <=1.0em, shorten >=1.0em, "\kappa"']
      \arrow[Rightarrow, from=botR, to=topR, shorten <=1.0em, shorten >=1.0em, "\beta"']
      & & 
      x' \arrow[""{name=bot, anchor=center, inner sep=0}, swap]{rr}{\id} & &
      x'
      \arrow[Rightarrow, from=bot, to=top, shorten <=1.0em, shorten >=1.0em, "\rho^{-1}\lambda"']
      \rlap{\, .}
    \end{tikzcd}
  \end{equation*}
  Coherence for $\M$ and for $J$ makes vertical composition of cells strictly associative and unital so we obtain the category $(\mathsf{H}J)_0=\A$, the category $(\mathsf{H}J)_1$, together with functors $\dom,\cod\colon(\mathsf{H}J)_1\to(\mathsf{H}J)_0$. Horizontal composition is associative and unital up to the invertible cells induced by the associators and unitors of $\M$, whence a double category.
\end{definition}

\begin{proposition}\label{prop:HHcocrossed}
If $J \colon \A \to \M$ is a bicategorical calibration, then $\mathsf{H}J$ is co-crossed.
\end{proposition}
\begin{proof}
  To show that $\cod$ is a fibration, let $p \in \A(x,z)$ and
  $n \in \M(y,z)$. By \ref{C3}, we can find $\tilde n \in \M(w,x)$,
  $\tilde p \in \A(w,y)$ and an invertible $2$-cell $\theta$ making
  the following a bipullback in $\M$:
  \begin{equation*}
    \begin{tikzcd}
      w \arrow{r}{\tilde n} \arrow[swap]{d}{\tilde p_*} &
      x \arrow{d}{p_*}\\
      y \arrow[swap]{r}{n} &
      z
      \arrow[from=1-1,to=2-2,phantom,"\cong_\theta"{description}]\rlap{ .}
    \end{tikzcd}
  \end{equation*}
  We claim that $\theta$ is $\cod$-cartesian as a cell of
  $\mathsf{H}J$. Indeed, let $\beta$ be a cell of $\mathsf{H}J$ as to
  the left of~\eqref{eq:HHcartexistence}, and write
  $\psi \defeq \alpha(\gamma m)\beta$ for the corresponding $2$-cell
  $n q_* \Rightarrow p_* (r_* m)$ in $\M$. By \ref{C2}, $p_*$ is a
  groupoid fibration, so by~\cite[Prop.~5.2(iii)]{Street2020} we may
  factor $\psi$ through the bipullback $\theta$ to obtain
  $s \in \M(u,w)$, an invertible
  $\tau \colon q_* \Rightarrow \tilde p_* s$ and
  $\sigma \colon \tilde n s \Rightarrow r_* m$ satisfying the
  left-hand equality below.

  Now by the existence half of \ref{C1} for $\tilde p$
  (Remark~\ref{rem:C1}), applied to $\tau$, we have $s' \in \A(u,w)$
  with $\tilde p s' = q$ and an invertible
  $\xi \colon s'_* \Rightarrow s$ with
  $(\tilde p_* \xi)\gamma = \tau$. So
  $\delta \defeq \sigma(\tilde n \xi)$ is a cell of $\mathsf{H}J$ as
  displayed below, and the right-hand equality says that
  $\beta = \theta\delta$ in $(\mathsf{H}J)_1$.
  \begin{equation}\label{eq:HHcartexistence}
    \begin{tikzcd}
      u \arrow[""{name=betatop, anchor=center, inner sep=0}]{r}{m} \arrow[swap]{dd}{q_*} &
      v \arrow[bend left=20]{dr}{r_*} \arrow[dd, swap, "(pr)_\ast" description] &[-2.5em] \\
      &&[-1em]
      x \arrow[bend left=20]{dl}{p_*} \\
      y \arrow[""{name=betabot, anchor=center, inner sep=0}, swap]{r}{n} &
      z &[-2.5em]
      \arrow[Rightarrow, from=betabot, to=betatop, shorten <=2.0em, shorten >=2.0em, "\beta"']
    \end{tikzcd}
    \quad=\quad
    \begin{tikzcd}
      u \arrow[""{name=sigmatop, anchor=center, inner sep=0}]{r}{m} \arrow[dotted]{d}{s} \arrow[""{name=taumid, anchor=center, inner sep=0}, bend right=50, swap]{dd}{q_*} &
      v \arrow{d}{r_*} \\
      w  \arrow[""{name=sigmabot, anchor=center, inner sep=0}]{r}{\tilde n} \arrow{d}{\tilde p_*}  &
      x \arrow{d}{p_*} \\
      y \arrow[swap]{r}{n} &
      z
      \arrow[from=2-1,to=3-2,phantom,"\cong_\theta"{description}]
      \arrow[Rightarrow, from=sigmabot, to=sigmatop, shorten <=1.0em, shorten >=1.0em, "\sigma"']
      \arrow[Rightarrow, from=taumid, to=2-1, shorten <=0.1em, "\tau"]
    \end{tikzcd}
    \quad=\quad
    \begin{tikzcd}
      u \arrow[""{name=deltatop, anchor=center, inner sep=0}]{r}{m} \arrow[dotted]{d}{s'_*} \arrow[bend right=50, swap]{dd}{q_*} &
      v \arrow{d}{r_*} \\
      w  \arrow[""{name=deltabot, anchor=center, inner sep=0}]{r}{\tilde n} \arrow{d}{\tilde p_*}  &
      x \arrow{d}{p_*} \\
      y \arrow[swap]{r}{n} &
      z
      \arrow[from=2-1,to=3-2,phantom,"\cong_\theta"{description}]
      \arrow[Rightarrow, from=deltabot, to=deltatop, shorten <=1.0em, shorten >=1.0em, "\delta"']
    \end{tikzcd}
  \end{equation}
  For uniqueness, suppose that $\delta'$ with $\dom(\delta') = t$ and
  $\cod(\delta') = r$ also satisfied $\theta\delta' = \beta$.
  By~\cite[Prop.~5.2(iii)]{Street2020}, again using \ref{C2}, there is
  a unique invertible $\nu \colon t_* \Rightarrow s'_*$ with
  $\delta' = \delta(\tilde n \nu)$ and $(\tilde p_* \nu)\gamma = \gamma$.
  By the uniqueness half of \ref{C1} for $\tilde p$, $t = s'$ and $\nu$
  is the identity, so $\delta' = \delta$. So $\beta$ factors uniquely
  through $\theta$, as required.

  This argument shows that every cell of $\mathsf{H}J$ which exhibits
  a bipullback in $\M$ is $\cod$-cartesian; conversely, any cartesian
  cell differs from such a bipullback by an invertible cell, and so is
  itself a bipullback. Since bipullback squares are closed under
  binary and nullary horizontal pasting, $\mathsf{H}J$ is co-crossed.
\end{proof}

\begin{remark}\label{rem:HHconverse}
  Unlike \cref{lemma:cocrossd-calib1}, the converse of
  \cref{prop:HHcocrossed} fails. Indeed, the proof above shows that,
  when $J$ is a bicategorical calibration, every $\cod$-cartesian cell
  of $\mathsf{H}J$ is filled by an invertible $2$-cell of $\M$.
  However, it is possible for $\mathsf{H}J$ to be co-crossed without
  this being the case, and then $J$ cannot be a bicategorical
  calibration.

  For a concrete counterexample, let $\M$ be the second oriental---the
  locally posetal $2$-category generated by morphisms
  $01 \colon 0 \rightarrow 1$, $12 \colon 1 \rightarrow 2$ and
  $02 \colon 0 \rightarrow 2$, and a $2$-cell
  $012 \colon 02 \Rightarrow 12 \circ 01$---and let
  $J \colon \A \rightarrow \M$ be the inclusion of the morphism
  $12 \colon 1 \rightarrow 2$. It is easy to see that this is not a
  calibration, since $12$ does not admit a bipullback along $02$ in
  $\M$.

  However, in $\mathsf{H}J$ the horizontal arrow $02$ \emph{does} admit a
  $\cod$-cartesian lift at $12$: the cell
  \begin{equation*}
    \begin{tikzcd}
      0 \arrow{r}{01} \arrow[swap]{d}{\id_\ast} &
      1 \arrow{d}{12_*}\\
      0 \arrow[swap]{r}{02} &
      2
      \arrow[from=1-1,to=2-2,phantom,"\Uparrow \scriptstyle 012"{description}]\rlap{ .}
    \end{tikzcd}
  \end{equation*}
  The remaining cartesian lifts are evident commuting squares, and a
  direct check shows that the cartesian cells contain the horizontal
  identities and are closed under horizontal composition.
\end{remark}

\begin{remark}
By \cref{prop:HHcocrossed}, a bicategorical calibration
$J \colon \A \rightarrow \M$ has a double category of corners
$\Cnr({\mathsf{H}J})$, in which:
\begin{itemize}[itemsep=0em]
\item Objects and vertical arrows are those of $\A$;
\item A horizontal arrow $x \tor y$ is a ``span''
  \begin{equation*}
    x \xleftarrow{\ \ m\ \ } w \xrightarrow{\ \ p\ \ } y
  \end{equation*}
  with $m \in \M(w,x)$ and $p \in \A(w,y)$; it is a ``span'' rather than a span
  because $m$ and $p$ live in different places. (Of course, there is
  an associated
  genuine span $(m, p_\ast)$ in $\M$.)
\item A cell bounded by horizontal arrows $(m,p)$ and $(m',p')$ and vertical
  arrows $f$ and $g$ is a pair
  \begin{equation}\label{eq:HHcnr-cell}
    \begin{tikzcd}
      x \arrow[swap]{d}{f_*} &
      w \arrow[""{name=top, anchor=center, inner sep=0}, swap]{l}{m} \arrow{d}{h_*} &&
      w \arrow{r}{p} \arrow[swap]{d}{h} &
      y \arrow{d}{g} \\
      x' &
      w' \arrow[""{name=bot, anchor=center, inner sep=0}]{l}{m'} &&
      w' \arrow[swap]{r}{p'} &
      y'\rlap{ ,}
      \arrow[Rightarrow, from=bot, to=top, shorten <=1.0em, shorten >=1.0em, "\kappa"']
    \end{tikzcd}
  \end{equation}
  where the left square is a $2$-cell of $\M$ and the right square
  commutes in $\A$. Vertical composition of cells is componentwise: by
  composition in $\A$ on the arrows $h$, and by vertical composition
  in $\mathsf{H}J$ on the $2$-cells $\kappa$.
\item The horizontal identity on $x$ is $(\id_x, \id_x)$;
\item Given horizontal arrows $(m,p) \colon x \tor y$ and
  $(n,q) \colon y \tor z$, their composite is the span
  $(m\tilde n, q\tilde p) \colon x \tor z$, wherein $\tilde n$ and
  $\tilde p$ are obtained via the chosen bipullback of $p_*$ along
  $n$:
\begin{equation*}
  \begin{tikzcd}[column sep=small]
    && u \arrow[swap]{dl}{\tilde n} \arrow{dr}{{\tilde p}_\ast} && \\
    & w \arrow[swap]{dl}{m} \arrow{dr}{p_\ast} &&
    v \arrow[swap]{dl}{n} \arrow{dr}{q_\ast} & \\
    x && y && z
    \arrow[from=1-3, to=3-3, phantom, "\cong_\theta"{description}]\rlap{ .}
  \end{tikzcd}
\end{equation*}
\end{itemize}
\end{remark}

\section{Second example: polynomials}\label{sec:polynomial_spans}

We now give our second main example of a double-categorical
distributive law---one whose composite is the double category of
\emph{polynomials} in a locally cartesian closed category $\E$. A
polynomial in this sense is a diagram in $\E$ of the shape:
\begin{equation}\label{eq:polynomial}
  \begin{tikzcd}
    &
    A \arrow[swap]{dl}{s} \arrow{r}{f} &
    B \arrow{dr}{t} & {}
    \\
    I & & & J \rlap{ ,}
  \end{tikzcd}
\end{equation}
where the intended interpretation of such a diagram is as the
\emph{polynomial functor}
\begin{equation*}
  \E / I \xrightarrow{\Delta_s} \E / A \xrightarrow{\Pi_f} \E / B \xrightarrow{\Sigma_t} \E / J\rlap{ .}
\end{equation*}
Here, for any arrow $f \colon X \to Y$ of $\E$, we write
\begin{itemize}[itemsep=0em]
\item $\Sigma_f \colon \E/X \to \E/Y$ for the functor which
  postcomposes with $f$;
\item $\Delta_f \colon \E/Y \to \E/X$ for the right adjoint of
  $\Sigma_f$ determined by pullback along $f$;
\item $\Pi_f \colon \E/X \to \E/Y$ for the further right
  adjoint of $\Delta_f$ given by local cartesian closedness.
\end{itemize}
In the case of the locally cartesian closed category of sets, the
polynomial functor associated to~\eqref{eq:polynomial} can be
described by the assignment:
\begin{equation*}
  (X_i \mid i \in I) \qquad \mapsto \qquad \Bigl(\,\textstyle\sum_{b \in B_j} \prod_{a \in A_{b}} X_{s(a)}
  \mid j \in J\,\Bigr)\rlap{ ;}
\end{equation*}
whose right-hand side is visibly a sum of monomials; whence 
the nomenclature. The systematic categorical treatment is due to
Gambino and Kock~\cite{GK2013}, who show that polynomials compose,
compatibly with the composition of the functors they present, and
assemble into a double category $\Poly(\E)$. It is this double
category which we seek to reconstruct. In more detail, we have:
\begin{definition}
  \label{def:poly-doublecat}
  If $\E$ is a locally cartesian closed category, then
  $\mathsf{Poly}(\E)$ is the double category in which:
  \begin{itemize}[itemsep=0em]
  \item The category of objects and vertical arrows is $\E$;
  \item Horizontal arrows from $I$ to $J$ are polynomial diagrams~\eqref{eq:polynomial};
  \item Cells bounded by horizontal arrows $(s,f,t)$ and
    $(s',f',t')$ and vertical arrows $p$ and $r$ are pairs of maps
    $(q,m)$ fitting into a commuting diagram:
    \begin{equation}\label{polycell}
      \begin{tikzcd}
        I \arrow[swap]{dd}{p} &
        A \arrow[swap]{l}{s} \arrow{r}{f} & 
        B \arrow{dd}{q} \arrow{r}{t} &
        J \arrow{dd}{r}\\
        & X \arrow{u}{m} \arrow{d}{} \arrow{ur}{} \pb{dr} & \\
        I' &
        A' \arrow{l}{s'} \arrow[swap]{r}{f'} &
        B' \arrow[swap]{r}{t'} &
        J'\rlap{ ;}
      \end{tikzcd}
    \end{equation}
  \item The horizontal identity on $I$ is the polynomial $(1,1,1)
    \colon I \leftarrow I \rightarrow I \rightarrow I$;
  \item The horizontal composite of polynomials $(s,f,t)$ and
    $(u,g,v)$ is determined by the upper part of the diagram
    \begin{equation} \label{polycomposition}
      \begin{tikzcd}
        &&& 
        N \ar[ddll] \ar[r] \pb{d} & 
        C' \arrow{dl}{\epsilon} \arrow{r}{k} \pb{ddr} \ar[dd] & 
        M \arrow{dd}{\Pi_g\Delta_u t} &\\
        &&& 
        B' \arrow[swap]{dl}{\pi} \arrow{dr}{\Delta_u t} \pb{dd}  &&& \\
        & A \arrow{dl}{s} \arrow[swap]{r}{f} & B \arrow[swap]{dr}{t} && 
        C \arrow{dl}{u} \arrow[swap]{r}{g} & 
        D \arrow[swap]{dr}{v} & \\
        I &&& J &&& K
      \end{tikzcd}
    \end{equation}
    where $\epsilon$ is the counit of the adjunction
    $\Delta_g \dashv \Pi_g$ at $\Delta_u t$.
  \end{itemize}
\end{definition}
We omit discussion of the coherence data and axioms for horizontal
composition in $\Poly(\E)$, as a direct approach is unenlightening.
The better treatment is that of~\cite[\S 2--3]{GK2013}, which
transports the horizontal structure of $\Poly(\E)$ from the equivalent
double category of polynomial functors between slices of $\E$. Our
approach, as foreshadowed in~\cref{sec:calibrated_cats}, instead
uses Ross's construction in~\cite{Street2020}, recast via
\cref{sec:calibr-bicat} in the language of distributive laws of double
categories.

We start from the bicategory $\Span(\E)$ of spans in $\E$, the
underlying bicategory of the double category of~\cref{cor:spans}. Its
objects are those of $\E$; its $1$-cells $x \to y$ are spans
$s \colon x \leftarrow w \rightarrow y \colon t$, composed by
pullback; and its $2$-cells $(s, t) \Rightarrow (s', t')$ are
maps of spans, that is, arrows $h \colon w \to w'$ with $s' h = s$
and $t' h = t$. Each arrow $f \colon x \to y$ of $\E$ gives a span
$f_* \defeq (\id_x, f)$, and this assignment provides the canonical
calibration of $\Span(\E)$:

\begin{proposition}
  \label{prop:spans-calibration}
  If $\E$ is locally cartesian closed, then the identity-on-objects
  pseudofunctor $(-)_* \colon \E \rightarrow \Span(\E)$ sending $f$
  to $(\id,f)$ is a bicategorical calibration.
\end{proposition}

This is much as in~\cite{Street2020}, but since our notion of
calibration is slightly different, we give a complete proof, although
one which defers to~\cite{Street2020} where possible.
\begin{proof}
  Write $\M \defeq \Span(\E)$. For \ref{C1}, we verify the conditions
  of \cref{rem:C1}. For existence, let $q \colon y \to z$ and $r \colon x \to z$ be arrows
  of $\E$ and let $p$ be a span
  $p_0 \colon x \leftarrow w \rightarrow y \colon p_1$ with
  $\tau \colon r_* \cong q_* p$. Since the left leg of $q_*$ is an
  identity, $q_* p$ is isomorphic to the span
  $p_0 \colon x \leftarrow w \rightarrow z \colon qp_1$. The left leg of
  $r_*$ is $\id_x$, so $\tau$ makes $p_0$ invertible, and then
  $p_0^{-1}$ is an isomorphism
  $\tau' \colon (p_1 p_0^{-1})_* \cong p$ with $q p_1 p_0^{-1} = r$
  and $(q_*\tau')\gamma = \tau$. Uniqueness follows
  since $(-)_*$ is locally fully faithful: a map of spans
  $(\id, k) \to (\id, k')$ is an arrow $h$ with $\id h = \id$ and
  $k'h = k$, forcing $h = \id$ and $k = k'$.
 
  For \ref{C2}, let $f \colon y \to z$ be an arrow of $\E$. For each
  object $x$, composition with $f_*$ sends a span
  $m_0 \colon x \leftarrow w \rightarrow y \colon m_1$ to the span
  $m_0 \colon x \leftarrow w \rightarrow z \colon fm_1$, up to
  isomorphism, so it is isomorphic to a discrete fibration
  $\M(x,y) \to \M(x,z)$; hence $f_*$ is a groupoid
  fibration~\cite[Rem.~4.2(a)]{Street2020}.

  For \ref{C3}, let $t \colon B \to J$ be an arrow of $\E$ and let
  $n$ be a span $g \colon D \leftarrow C \rightarrow J \colon u$.
  Writing $h^* \defeq (h, \id)$ for the span which reverses an arrow
  $h$ of $\E$, we have $n \cong u_* g^*$. We follow Ross's proof
  of~\cite[Prop.~8.3]{Street2020}. First, form the pullback
  $\pi \colon B \leftarrow B' \rightarrow C \colon \Delta_u t$ of $t$
  and $u$, the square on the right of the left-hand diagram below. Now
  let $k \colon C' \to M$ be the pullback of
  $\Pi_g \Delta_u t \colon M \to D$ along $g$, so that the other
  projection $C' \to C$ is $\Delta_g \Pi_g \Delta_u t$, and let
  $\epsilon \colon C' \to B'$ be the counit at $\Delta_u t$ of the
  adjunction $\Delta_g \dashv \Pi_g$; these form the rectangle on the
  left of the same diagram.
  \begin{equation*}
    \begin{tikzcd}
      M \arrow[swap]{d}{\Pi_g \Delta_u t} &
      C' \arrow[swap]{l}{k} \arrow{r}{\epsilon} &
      B' \arrow{r}{\pi} \arrow[swap]{d}{\Delta_u t} \pb{dr} &
      B \arrow{d}{t} \\
      D & &
      C \arrow{ll}{g} \arrow[swap]{r}{u} &
      J
    \end{tikzcd}
    \qquad\qquad
    \begin{tikzcd}[column sep=5em]
      M \arrow{r}{\epsilon_* k^*} \arrow[swap]{d}{(\Pi_g \Delta_u t)_*} &
      B' \arrow{r}{\pi_*} \arrow{d}{(\Delta_u t)_*} &
      B \arrow{d}{t_*} \\
      D \arrow[swap]{r}{g^*} &
      C \arrow[swap]{r}{u_*} &
      J\rlap{ .}
      \arrow[from=1-1,to=2-2,phantom,"\cong"{description}]
      \arrow[from=1-2,to=2-3,phantom,"\cong"{description}]
    \end{tikzcd}
  \end{equation*}
  By~\cite[Prop.~5.4]{Street2020} the pullback square on the left gives
  the right-hand bipullback square on the right; and
  by~\cite[Prop.~7.1]{Street2020} the rectangle on the left gives the
  left-hand one. Pasting these two squares, and using
  $n \cong u_* g^*$, gives a bipullback of $t_*$ along $n$ whose legs
  are $\tilde n = (\pi\epsilon)_* k^* \colon M \to B$ and
  $(\Pi_g \Delta_u t)_*$, as \ref{C3} requires.
\end{proof}

\begin{proposition}\label{prop:poly}
  If $(-)_* \colon \E \to \Span(\E)$ is the calibration
  of~\cref{prop:spans-calibration}, then there is an
  identity-on-objects equivalence
  $\Cnr(\mathsf{H}((-)_*)) \simeq \Poly(\E)$.
\end{proposition}
\begin{proof}
  Both double categories have $\E$ as their category of objects and
  vertical arrows. Next, a horizontal arrow $I \tor J$ of
  $\Cnr(\mathsf{H}((-)_*))$ comprises a $1$-cell
  $f \colon B \leftarrow A \rightarrow I \colon s$ of $\Span(\E)$
  and a map $t \colon B \to J$ of $\E$: which is precisely a
  polynomial~\eqref{eq:polynomial}.

  A cell of $\Cnr(\mathsf{H}((-)_*))$ with the same boundary as
  in~\eqref{polycell} comprises a square in $\Span(\E)$ and a
  commuting square in $\E$, as in~\eqref{eq:HHcnr-cell}:
  \begin{equation*}
    \begin{tikzcd}
      I \arrow[swap]{d}{p_*} &
      B \arrow[""{name=top, anchor=center, inner sep=0}, swap]{l}{(f,s)} \arrow{d}{q_*} &&
      B \arrow{r}{t} \arrow[swap]{d}{q} &
      J \arrow{d}{r} \\
      I' &
      B' \arrow[""{name=bot, anchor=center, inner sep=0}]{l}{(f',s')} &&
      B' \arrow[swap]{r}{t'} &
      J'\rlap{ .}
      \arrow[Rightarrow, from=bot, to=top, shorten <=1.0em, shorten >=1.0em, "\kappa"']
    \end{tikzcd}
  \end{equation*}
  The square in $\E$ is the right-hand square of~\eqref{polycell}.
  As for $\kappa$, the composite span $(f',s')\,q_*$ is formed by the
  pullback on the left below, so $\kappa$ is a map of spans $m$ as on
  the right:
  \begin{equation*}
    \begin{tikzcd}
      X \arrow{r}{\pi_2} \arrow[swap]{d}{\pi_1} \pb{dr} &
      A' \arrow{d}{f'} \\
      B \arrow[swap]{r}{q} &
      B'
    \end{tikzcd}
    \qquad\qquad
    \begin{tikzcd}
      & A \arrow[swap]{dl}{f} \arrow{dr}{ps} & \\
      B &
      X \arrow{u}{m} \arrow{l}{\pi_1} \arrow[swap]{r}{s'\pi_2} &
      I'\rlap{ .}
    \end{tikzcd}
  \end{equation*}
  These are precisely the data and conditions of the left-hand part
  of~\eqref{polycell}.

  Horizontal identities agree, and so do composites of horizontal
  arrows. The composite of $(s,f,t)$ and $(u,g,v)$ in
  $\Cnr(\mathsf{H}((-)_*))$ is built from the bipullback of $t_*$
  along the span $(g,u)$, which the proof of
  \cref{prop:spans-calibration} constructs as the upper part
  of~\eqref{polycomposition}, with $\tilde n$ the span
  $k \colon M \leftarrow C' \rightarrow B \colon \pi \epsilon$.
  Composing with $(f,s)$ and $v$ gives the apex $N$ and the arrow
  $v \cdot \Pi_g\Delta_u t$ of~\eqref{polycomposition}. That is, the
  distributive law sends the cospan
  $t \colon B \rightarrow J \leftarrow D \colon (g,u)$ to the span
  $(k,\pi\epsilon) \colon B \leftarrow M \rightarrow D \colon
  {\Pi_g\Delta_u t}$. The remaining structure---vertical and
  horizontal composition of cells, and the coherence cells---is on
  both sides induced by the universal properties of pullback and
  dependent product, and we leave checking that the two agree to the
  reader.
\end{proof}

Ross's two other main examples~\cite{Street2020}, relations and
modules, are also calibrations in our sense.

\begin{example}\label{ex:rel}
  Let $\C$ be a regular category in which, for each arrow
  $h \colon y \to x$, the pullback map
  $h^{-1} \colon \mathrm{Sub}\,x \to \mathrm{Sub}\,y$ has a right
  adjoint $\forall_h$; for instance, $\C$ could be a topos. Let $\A$
  be the category of objects and monomorphisms of $\C$, and let
  $\M = \Rel(\C)$ be the locally ordered bicategory of relations in
  $\C$, whose $1$-cells from $x$ to $y$ are the jointly monic spans
  $m_0 \colon x \leftarrow w \rightarrow y \colon m_1$, composed by
  pullback followed by image factorisation. We claim that the
  identity-on-objects functor $J = (-)_*$, which sends each
  monomorphism $p \colon x \to y$ to its graph $(\id_x, p)$, is a
  calibration. Throughout, we use that for a monomorphism $p$,
  composing a relation $(m_0, m_1)$ with $p_*$ gives $(m_0, pm_1)$,
  with no image factorisation needed.

  For \ref{C1}, note that $J$ is locally fully faithful: an
  inequality $k_* \le k'_*$ is witnessed by an arrow $h$ with
  $\id h = \id$ and $k'h = k$, forcing $k = k'$. So by
  \cref{rem:C1}, and since $\Rel(\C)$ is locally ordered, we need
  only show that if $q \colon y \to z$ and $r \colon x \to z$ are
  monomorphisms and $p = (p_0, p_1)$ is a relation with
  $q_* p = r_*$, then $p$ is the graph of a monomorphism. Now
  $q_* p = (p_0, qp_1)$, which can equal $(\id_x, r)$ only if $p_0$ is
  invertible; whereupon $p = (p_1 p_0^{-1})_*$, and $p_1 p_0^{-1}$ is
  monic as $q p_1 p_0^{-1} = r$ is.

  For \ref{C2}, let $p \colon y \to z$ be a monomorphism. For each
  object $x$, composition with $p_*$ is the map
  $\M(x,y) \to \M(x,z)$ sending $(m_0, m_1)$ to $(m_0, pm_1)$, which
  is a discrete fibration of posets; hence $p_*$ is a groupoid
  fibration~\cite[Def.~3.3]{Street2020}.

  For \ref{C3}, let $p \colon x \to z$ be a monomorphism and
  $n = (n_0, n_1)$ a relation from $y$ to $z$. Let
  $\tilde p \colon w \to y$ be the inclusion of the subobject
  $\forall_{n_0} n_1^{-1}(x)$---the largest subobject of $y$ whose
  image under $n$ lies in $x$---and let $\tilde n$ be the restriction
  of $n$ to it, regarded as a relation from $w$ to $x$
  (cf.~\cite[Ex.~10.3]{Street2020}). This is the bipullback of
  $p_*$ along $n$ constructed in~\cite[Thm.~9.3,
  Prop.~10.2]{Street2020}.

  The double category of corners of $\mathsf{H}J$ then has Ross's
  polynomials in $\Rel(\C)$ as its horizontal arrows.
\end{example}

\begin{example}\label{ex:mod}
  Let $\M = \Mod$ be the bicategory of small categories and modules,
  whose morphisms from $K$ to $B$ are functors
  $B^\op \times K \to \Set$, and whose composition is by
  coend~\cite[Ex.~2.2]{Street2020}. Let $\A = \DFib$ be the category
  of small categories and discrete fibrations. We claim that the
  identity-on-objects pseudofunctor $(-)_*$ which sends each discrete
  fibration $f \colon K \to B$ to the module $f_*(b,k) = B(b,fk)$ is a
  calibration.

  Throughout, we use the following description of composition with
  $q_*$ for a discrete fibration $q \colon E \to B$. Let $z$ be the
  presheaf on $B$ corresponding to $q$, so that $E$ is its category of
  elements; write $\widehat B \defeq [B^\op, \Set]$, and
  $y \colon B \to \widehat B$ for the Yoneda embedding. For each small
  category $K$ we have
  $\M(K,E) \simeq [K, \widehat E] \simeq [K, \widehat B/z] \cong [K,
  \widehat B]/\Delta z$, and under these equivalences, composition
  with $q_*$ is isomorphic to the projection functor
  $[K, \widehat B]/\Delta z \to [K, \widehat B]$. In particular, it is
  a groupoid fibration, whence each $q_\ast$ is a groupoid
  fibration~\cite[Def.~3.3]{Street2020}, which verifies \ref{C2}.

  For \ref{C1}, bipullbacks are invariant under equivalence, so we
  may compute the one in \ref{C1} with composition with $q_*$
  replaced by the projection above. This is a discrete fibration, and
  so an isofibration, so the bipullback may be computed as an
  ordinary pullback. An object of that pullback is a discrete
  fibration $r \colon K \to B$ together with a natural transformation
  $y r \To \Delta z$; by the Yoneda lemma, that is a functor
  $s \colon K \to E$ with $qs = r$, and $s$ is a discrete fibration
  because $q$ and $qs$ are. The pullback is therefore the discrete
  category of such $s$, and the comparison from $\A(K,E)$ is an
  isomorphism.

  For \ref{C3}, let $p \colon E \to B$ be a discrete fibration with
  corresponding presheaf $z$ as above, and let $n$ be a module from
  $K$ to $B$. Let $\tilde p \colon P \to K$ be the projection from the
  category of elements of the presheaf
  $k \mapsto \widehat B(n(-,k), z)$ on $K$, which is the right lifting
  of $z$ through $n$. Ross's construction of bipullbacks in the proof
  of~\cite[Thm.~9.3]{Street2020} gives a bipullback of $p_*$ along $n$
  whose leg into $K$ is the tabulation of this right lifting, namely
  $\tilde p_*$~\cite[Prop.~10.4, Cor.~10.5]{Street2020}.

  The double category of corners of $\mathsf{H}((-)_*)$ then has the
  polynomials in $\Mod$ of~\cite[\S10]{Street2020} as its horizontal
  arrows. By~\cite[Prop.~8.6]{Street2020},
  each such polynomial $(m, p) \colon X \tor Y$ induces a functor
  $\widehat X \to \widehat Y$: right lifting through $m$, followed by
  composition with $p_*$.
\end{example}

\bibliographystyle{alpha}
\bibliography{multicat}

\end{document}